\documentclass[12pt]{article}%
\usepackage{pgf}
\usepackage{tikz}
\usepackage{color}
\usepackage{amsmath}
\usepackage{amsfonts}
\usepackage{amssymb}
\usepackage{graphicx}
\usepackage{amsmath,amssymb,amsthm,amsfonts}
\usepackage{tabularx}
\usepackage{latexsym}
\usepackage{amssymb}
\usepackage{amsfonts}
\usepackage{mathrsfs}
\usepackage{chemarrow}
\usepackage{amsthm,array,bm}
\usepackage{pgf}
\usepackage{pgflibrarysnakes}
\usepackage{tikz}
\usepackage{mathtools}
\usepackage{array}
\usepackage{color,graphicx}
\usepackage{pstricks}
\usepackage{amsmath}
\usepackage{float}
\providecommand{\U}[1]{\protect\rule{.1in}{.1in}}
\def\vs{\vskip.3cm}
\def\br{\mathbb R}
\def\bz{\mathbb Z}
\def\bc{\mathbb C}
\def\wt{\widetilde}

\def\id{\text{\rm Id\,}}

\def\cV{\mathcal V}

\def\noi{\noindent}
\def\gdeg{G\text{\rm -deg}}
\def\vp{\varphi}

\newcommand{\amal}[5]{#1\prescript{#4}{}\times^{#5}_{#3}#2}

\newcommand{\eqdeg}[1]{#1\mbox{\rm -}\deg}

\newtheorem{theorem}{Theorem}[section]

{\bfseries}{\rmfamily}

\newtheorem{remark}[theorem]{Remark}

\newtheorem{remark-definition}[theorem]{Remark and Definition}
\font\smc=cmcsc10
\newrgbcolor{lightblueA}{.50 1 1}
\newrgbcolor{violet}{.6 .1 .8}
\newrgbcolor{lightyellow}{1 1 .8}
\newrgbcolor{lightblue}{.80 1 1}
\newrgbcolor{mygreen}{0 .66 .05}
\definecolor{mygreen}{rgb}{0,.66,.05}
\definecolor{lightyellow}{rgb}{1,1,.80}
\newrgbcolor{orange}{1 .6 0}
\newrgbcolor{GreenYellow}{.85 1 .31}
\newrgbcolor{Yellow}{1  1  0}
\newrgbcolor{Goldenrod}{1  .90  .16}
\newrgbcolor{Dandelion}{1  .71  .16}
\newrgbcolor{Apricot}{1  .68  .48}
\newrgbcolor{Peach}{1  .50  .30}
\newrgbcolor{Melon}{1  .54  .50}
\newrgbcolor{YellowOrange}{1  .58  0}
\newrgbcolor{Orange}{1  .39  .13}
\newrgbcolor{BurntOrange}{1  .49  0}
\newrgbcolor{Bittersweet}{1.  .4300  .24}
\newrgbcolor{RedOrange}{1  .23  .13}
\newrgbcolor{Mahogany}{1.  .4475  .4345}
\newrgbcolor{Maroon}{1.  .4084  .5376}
\newrgbcolor{BrickRed}{1.  .3592  .3232}
\newrgbcolor{Red}{1  0  0}
\newrgbcolor{OrangeRed}{1  0  .50}
\newrgbcolor{RubineRed}{1  0  .87}
\newrgbcolor{WildStrawberry}{1  .04  .61}
\newrgbcolor{Salmon}{1  .47  .62}
\newrgbcolor{CarnationPink}{1  .37  1}
\newrgbcolor{Magenta}{1  0  1}
\newrgbcolor{VioletRed}{1  .19  1}
\newrgbcolor{Rhodamine}{1  .18  1}
\newrgbcolor{Mulberry}{.6668  .1180  1.}
\newrgbcolor{RedViolet}{.9538  .4060  1.}
\newrgbcolor{Fuchsia}{.5676  .1628  1.}
\newrgbcolor{Lavender}{1  .52  1}
\newrgbcolor{Thistle}{.88  .41  1}
\newrgbcolor{Orchid}{.68  .36  1}
\newrgbcolor{DarkOrchid}{.60  .20  .80}
\newrgbcolor{Purple}{.55  .14  1}
\newrgbcolor{Plum}{.50  0  1}
\newrgbcolor{Violet}{.98 .15 .95}
\newrgbcolor{RoyalPurple}{.25  .10  1}
\newrgbcolor{BlueViolet}{.84  .38  .98}
\newrgbcolor{Periwinkle}{.43  .45  1}
\newrgbcolor{CadetBlue}{.38  .43  .77}
\newrgbcolor{CornflowerBlue}{.35  .87  1}
\newrgbcolor{MidnightBlue}{.4414  .9259  1.}
\newrgbcolor{NavyBlue}{.06  .46  1}
\newrgbcolor{RoyalBlue}{0  .50  1}
\newrgbcolor{Blue}{0  0  1}
\newrgbcolor{Cerulean}{.06  .89  1}
\newrgbcolor{Cyan}{0  1  1}
\newrgbcolor{ProcessBlue}{.04  1  1}
\newrgbcolor{SkyBlue}{.38  1  .88}
\newrgbcolor{Turquoise}{.15  1  .80}
\newrgbcolor{TealBlue}{.1572  1.  .6668}
\newrgbcolor{Aquamarine}{.18  1  .70}
\newrgbcolor{BlueGreen}{.15  1  .67}
\newrgbcolor{Emerald}{0  1  .50}
\newrgbcolor{JungleGreen}{.01  1  .48}
\newrgbcolor{SeaGreen}{.31  1  .50}
\newrgbcolor{Green}{0  1  0}
\newrgbcolor{ForestGreen}{.1992  1.  .2256}
\newrgbcolor{PineGreen}{.3100  1.  .5575}
\newrgbcolor{LimeGreen}{.50  1  0}
\newrgbcolor{YellowGreen}{.56  1  .26}
\newrgbcolor{SpringGreen}{.74  1  .24}
\newrgbcolor{OliveGreen}{.6160  1.  .4300}
\newrgbcolor{RawSienna}{.53  .28  .16}
\newrgbcolor{Sepia}{1.  .7510  .70}
\newrgbcolor{Brown}{.41  .25  .18}
\newrgbcolor{TAN}{.86  .58  .44}
\newrgbcolor{Gray}{1.  1.  1.}
\newrgbcolor{Black}{1  1  1}
\newrgbcolor{White}{1  1  1}
\def\cV{\mathcal{V}}

\newcolumntype{Y}{>{\raggedleft\arraybackslash}X}
\begin{document}

\title{Global Nonlinear Normal Modes in\\ the Fullerene Molecule $C_{60}$}
\author{ Carlos Garc\'{\i}a-Azpeitia\thanks{
Departamento de Matem\'{a}ticas, Facultad de Ciencias, Universidad Nacional
Aut\'{o}noma de M\'{e}xico, 04510 M\'{e}xico}, Wieslaw Krawcewicz\thanks{Center for Applied Mathematics at Guangzhou University, Guangzhou, 510006 China. } \footnotemark[4]
\and Manuel Tejada-Wriedt\footnotemark[1] \thanks{Laboratorios de Nanomedicina. Nano Tutt S.A. de C.V. Tripoli 803, Portales, 03300, Mexico City, Mexico} , Haopin Wu\thanks{ Department of Mathematical Sciences University of Texas at Dallas Richardson, 75080 USA }} 
\maketitle

\begin{abstract}

\noi In this paper we analyze nonlinear dynamics of  the fullerene molecule. We prove 
the  existence of global branches of periodic solutions  emerging from an icosahedral equilibrium   (nonlinear normal modes). We also determine the symmetric properties of the
branches of nonlinear normal modes for maximal orbit types.
 We find several 
solutions  which are standing and rotating waves that propagate along the
molecule with icosahedral, tetrahedral, pentagonal and triangular symmetries.
  We complement our theoretical results with numerical computations using Newton's method.

\end{abstract}

\section{Introduction}\label{sec:1}

The fullerene molecule was discovered in 1986 by R. Smalley, R. Curl, J.
Heath, S. O'Brien, and H. Kroto, and since then continues to attract great
deal of attention in the scientific community (e.g. the original work
\cite{Kr} has currently 15 thousand citations). Its importance led in 1996 to
awarding the Nobel prize in chemistry to Smalley, Curl and Kroto. Since
its discovery, applications of fullerene C60 have been extensively explored in biomedical research due to their unique structure and physicochemical properties.

\begin{figure}[ptb]
\begin{center}
\resizebox{11cm}{!}{\includegraphics{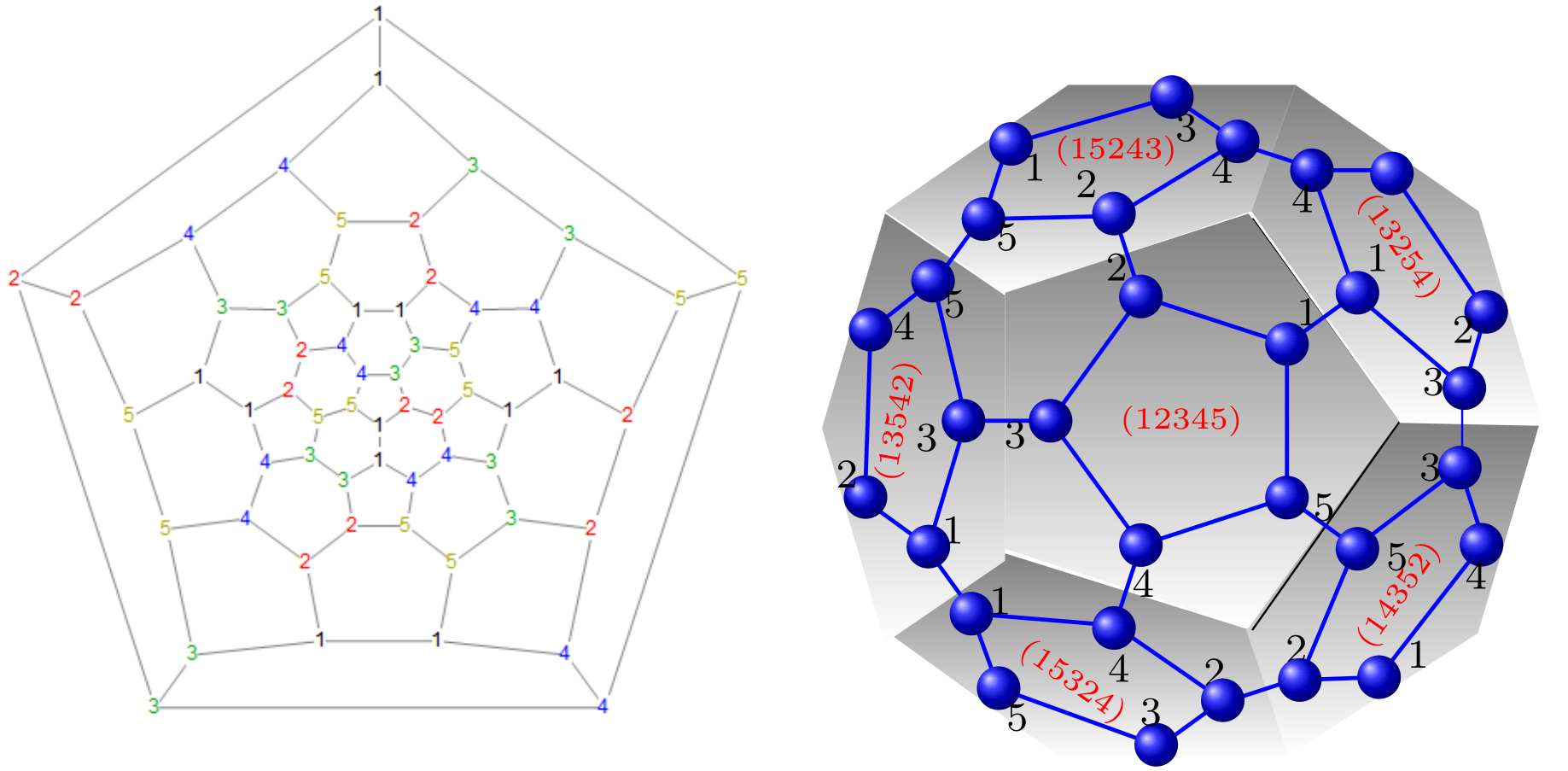} } 
\end{center}
\caption{The fullerene molecule}%
\label{fig:fullerene}%
\end{figure}

The fullerene molecule is composed of $60$ carbon atoms at the vertices of a
truncated icosahedron (see Figure \ref{fig:fullerene}). Various mathematical
models for the fullerene molecule are built in the framework of classical
mechanics. Many force fields have been proposed for the fullerene in terms of
bond stretching, bond bending, torsion and van der Waals forces. Some force
fields were optimized to duplicate the normal modes obtained using IR or Raman
spectroscopy, but only few of these models reflect the nonlinear
characteristics of the fullerene. For instance, the force field implemented in
\cite{Wa} for carbon nanotubes and in \cite{Be} for the fullerene, assumes
bond deformations that exceed very small fluctuations about equilibrium
states, while the force fields proposed in \cite{We} and \cite{Wu} are
designed only to consider small fluctuations in the fullerene model. Although
the linear vibrational modes of fullerene can be measured using IR and Raman
spectroscopy, increasing attention has been given to the mathematical study of
other vibrational modes that cannot be measured experimentally (see
\cite{He,Ho,Ji,Wa,We,Wu,Sc} and the large bibliography therein).

\paragraph{Mathematical Models.}

In molecular dynamics, a fullerene molecule model consists of a Newtonian
system
\begin{equation}
\label{eq:1}m\ddot u=-\nabla V(u),
\end{equation}
where the vector $u(t)\in\widetilde{\mathscr V}:=\mathbb{R}^{180}$ represents
the positions of 60 carbon atoms in space, $m$ is the carbon mass (which by
rescaling the model can be assumed to be one) and $V (u)$ is the energy given
by a force field. This force field is symmetric with respect to the action of the
group $\Gamma:=I\times O(3)$, where $I$ denotes the full icosahedral group.
This important property allows application of various equivariant methods to
analyze its dynamical properties. In order to make the system \eqref{eq:1}
reference point-depended, we define the subspace $\mathscr V$ of
$\mathbb{R}^{180}$ by
\begin{equation}
\label{eq:space}\mathscr V:=\{x=(x_{1},x_{2},\dots, x_{60}): x_{k}%
\in\mathbb{R}^{3},\; \sum_{k=1}^{60} x_{k}=0\},
\end{equation}
from which we exclude the collision orbits, i.e. we consider the restriction
of the system \eqref{eq:1} to the set $\Omega_{o}:=\{x\in\mathscr V:
x_{j}\not = x_{k}, \; j\not =k\}$.

\paragraph{Analysis of Nonlinear Molecular Vibrations.}

The mathematical analysis of a molecular model includes two objectives:
identification of the \textit{normal frequencies} and the classification of different
families of \textit{periodic solutions with various spatio-temporal symmetries}
emerging from the equilibrium configuration of the
molecule (nonlinear normal modes). Let us emphasize that the classification of normal modes
 is a central problem of the molecular spectroscopy. 

The study of
periodic orbits in Hamiltonian systems can be traced back to Poincare and
Lyapunov, who proved the existence of nonlinear normal modes (periodic orbits)
near an elliptic equilibrium under non-resonant conditions. Later on,
Weinstein (cf. \cite{Weinstein}) extended this result to the case with
resonances. However, in general, the existence of nonlinear normal modes is
not guaranteed under the presence of resonances. Indeed, Moser presented in
\cite{Mos} an example of a Hamiltonian systems with resonances, where the
linear system is full of periodic solutions, but the nonlinear system has
none. 

Let us point out that due to the icosahedral symmetries, the fullerene molecule
$C_{60}$ has resonances of multiplicities 3, 4 and 5, i.e. the existence of linear modes does not guarantee the
existence of nonlinear normal modes in fullerene $C_{60}$ due to resonances. Therefore, one needs
a good method that takes in consideration these symmetries. Standard methods
for such analysis may use reductions to the $H$-fixed point spaces, normal
form classification, center manifold theorem, averaging method and/or
Lyapunov--Schmidt reduction.

\paragraph{Variational Reformulation.}

The problem of finding periodic solutions for the fullerene can be
reformulated as a variational problem on the Sobolev space $H_{2\pi}%
^{1}(\mathbb{R};{\mathscr V})$ (of $2\pi$-periodic ${\mathscr V}$-valued
functions) with the functional
\[
J_{\lambda}(u):=\int_{0}^{2\pi}\left[  \frac{1}{2}|\dot{u}(t)|^{2}-\lambda
^{2}V(u(t))\right]  dt,\quad u\in H_{2\pi}^{1}(\mathbb{R};\Omega_{o}),
\]
where $V$ is the force field, $\lambda^{-1}$ the frequency and $u$ the
renormalized $2\pi$-periodic solution. The existence of periodic solutions
(with fixed frequency $\lambda^{-1}$) is equivalent to the existence of
critical points of $J_{\lambda}$. It follows from the construction of the
force field that the functional $J_{\lambda}$ is invariant under the action of
the group
\[
G:=\Gamma\times O(2)=(I\times O(3))\times O(2),
\]
which acts as permutations of atoms, rotations in space, and translations and
reflection in time, respectively.

\begin{figure}[ptb]

\begin{center}
\resizebox{7cm}{!}{\includegraphics{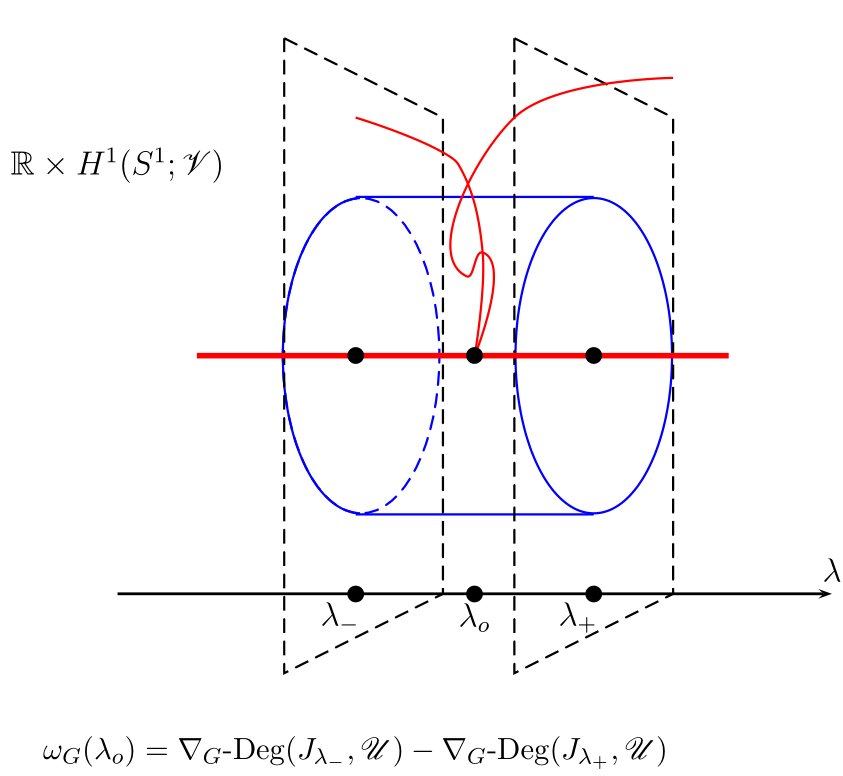} } 
\end{center}
\caption{Local bifurcation.}\label{fig:bif}
\end{figure}
\paragraph{Gradient Equivariant Degree Method.}

To provide an alternative to the equivariant singularity theory (cf.
\cite{Golubitsky}) and other geometric methods that have been used to analyze molecules (see \cite{Hoy3},
\cite{Mon1}, \cite{Mon2} and references), we proposed (see
\cite{GaTe16,BeKr,BeGa}) the method based on the equivariant gradient degree
(fundamental properties of the gradient equivariant degree are collected in
Appendix \ref{sec:equi-degree}) -- a generalization of the
Brouwer/Leray-Schauder degree that was developed in \cite{Geba} for the
gradient maps (see also \cite{DaKr} and \cite{RY2}). The gradient equivariant
degree is just one of many equivariant degrees that were introduced in the
last three decades for various types of differential equations (see
\cite{survey}, \cite{BaKr06}, \cite{IzVi03}, \cite{KW} and references therein).

To describe the main idea of this method, let us point out that the gradient
equivariant degree satisfies all the standard properties expected from a
degree theory (i.e. existence, additivity, homotopy and multiplicativity
properties). The $G$-equivariant gradient degree $\nabla_{G}\text{-Deg}(\nabla
J_{\lambda},\mathscr U)$ of $\nabla J_{\lambda}$ on $\mathscr U$ can be
expressed elegantly as an element of the Euler ring $U(G)$ (which is the free
$\mathbb{Z}$-module generated by the conjugacy classes $(H)$ of closed
subgroups $H\leq G$) in the form
\[
\nabla_{G}\text{-Deg}(\nabla J_{\lambda},\mathscr U)=n_{1}(H_{1})+n_{2}%
(H_{2})+\dots+n_{m}(H_{m}),\;\;\;n_{k}\in\mathbb{Z},
\]
where $\mathscr U$ is a neighborhood of the $G$-orbit of the equilibrium
$u_{o}$ (for some non-critical frequency $\lambda^{-1}$) and $(H_{j})$ are the
orbit types in $\mathscr U$. The changes of $\nabla_{G}\text{-deg}(\nabla
J_{\lambda},\mathscr U)$ when $\lambda^{-1}$ crosses a critical frequency
$\lambda_{o}^{-1}$ allow to establish the existence of various families of
orbits of periodic molecular vibrations and their symmetries emerging from an equilibrium. In fact, the equivariant topological invariant
\begin{equation}\label{eq:top-inv}
\omega_{G}(\lambda_{o}):=\nabla_{G}\text{-Deg\thinspace}(\nabla J_{\lambda
_{-}},\mathscr U)-\nabla_{G}\text{-Deg\thinspace}(\nabla J_{\lambda_{+}%
},\mathscr U)
\end{equation}
provides a full topological
characterization of the families of periodic solutions (together with their
symmetries) emerging from an equilibrium at $\lambda_{o}$ (cf. \cite{DGR}).
More precisely, for every non-zero coefficient $m_{j}$ in
\[
\omega_{G}(\lambda_{o})=m_{1}(K_{1})+m_{2}(K_{2})+\dots m_{r}(K_{r}),
\]
there exists a global family of periodic molecular vibrations with symmetries
at least $K_{j}$ (see Figure \ref{fig:bif} below). Moreover, if $(K_{j})$ is a maximal
orbit type then this family has exact symmetries $K_{j}$.

\paragraph{Global Bifurcation Result.}

The so-called classical Rabinowitz Theorem \cite{Ra} establishes occurrence of
a global bifurcation from purely local data for compact perturbations of the
identity. Its main idea is that if the maximal connected set $\mathcal{C}$
bifurcating from a trivial solution is compact (i.e. bounded), then the sum of
the local Leray-Schauder degrees at the set of bifurcation points of
$\mathcal{C}$ is zero. Since such maximal connected set $\mathcal{C}$ is
either unbounded or comes back to another bifurcation point, this result is
also referred to as the \emph{global Rabinowitz alternative} (we refer to
Nirenberg's book \cite{Ni} where a simplified proof of this statement is
presented in Theorem 3.4.1).

The classical Rabinowitz's global bifurcation argument can be easily adapted
in the equivariant setting for the gradient $G$-equivariant degree (cf.
\cite{GolRyb}). That is, for any $G$-orbit of a compact (bounded) branch
$\mathcal{C}$ in $\mathbb{R}_{+}\times H_{2\pi}^{1}(\mathbb{R};\Omega_{o})$
containing $(\lambda_{0},u_{o})$ we have
\begin{equation}
\sum_{k=0}^{m}\omega_{G}(\lambda_{k})=0,\label{eq:int-global}%
\end{equation}
(see Figure \ref{fig:glob-bif}), where $\lambda_{k}^{-1}$ are the normal modes
belonging to $\mathcal{C}$. In this context the \textbf{global property} means
that a family of periodic solutions, represented by continuous branch
$\mathcal{C}$ in $\mathbb{R}_{+}\times H_{2\pi}^{1}(\mathbb{R};\Omega_{o})$,
is not compact or comes back to another bifurcation point from the equilibrium.
The non-compactness of $\mathcal{C}$ implies that the norm or period of
solutions from $\mathcal{C}$ goes to the infinity, $\mathcal{C}$ ends in a
collision orbit or goes to a different equilibrium point.

\begin{figure}
\begin{center}
\resizebox{11cm}{!}{\includegraphics{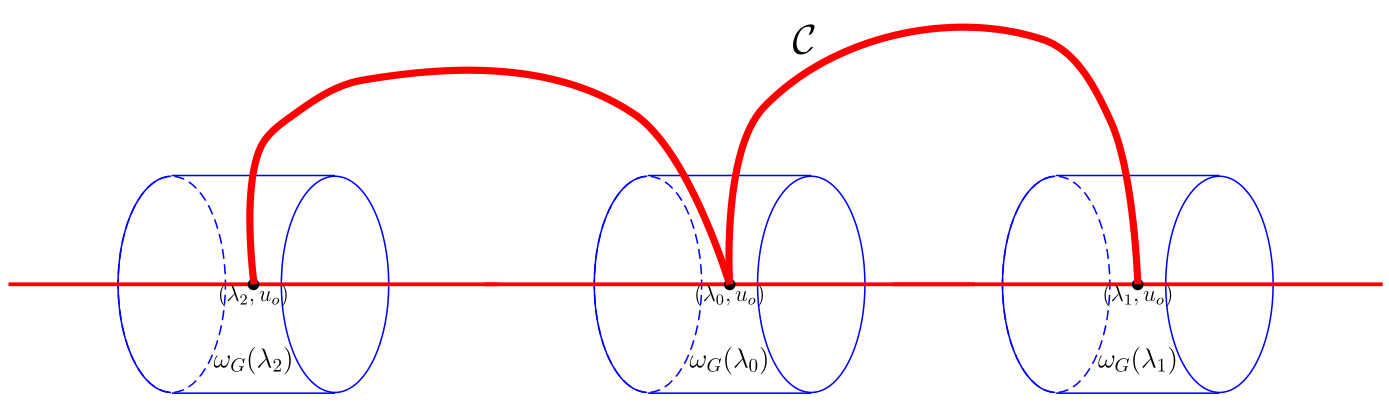} } 
\end{center}

\caption{Global bifurcation}\label{fig:glob-bif}
\end{figure}

By applying formula \eqref{eq:int-global} one can establish an effective
criterium allowing to determine the existence of the non-compact
branches of nonlinear normal modes with particular (e.g. maximal) orbit types.
To be more precise, it is sufficient to consider all the critical frequencies
$\lambda_{k}^{-1}$ corresponding to the first Fourier mode and simply show,
that for some of them, say $\lambda_{0}^{-1}$, the sum in
\eqref{eq:int-global} can never be zero. For the fullerene molecule such non-compact global
branches exist.

\paragraph{Novelty.}

In this paper, we apply equivariant gradient degree 
to the classification of the global nonlinear modes in a model of the
fullerene molecule $C_{60}$. By taking advantage of various properties of the gradient
equivariant degree, the approximate values (obtained numerically) of the
normal frequencies can be used to determine
(under plausible isotypical non-resonance assumption) the exact values of the topological
invariants $\omega_{G}(\lambda_{o})$. In particular, the information contained
in the topological invariants can be applied to obtain the presence of such
global branches of periodic solutions with the maximal orbit types, as it is
presented in our main theorem (Theorem \ref{th:main}). 

Let us point out that (to the best of our
knowledge) all the previous studies of the fullerene molecule have considered
only the existence of the linear modes (which have constant frequency), while the
occurrence of the non-linear normal modes have frequencies depending on
the amplitudes of the oscillation. \emph{Such an
analysis of nonlinear normal modes for the fullerene molecule was never
done before}. We complement our results with numerical computations using Newton's method
and pseudo-arclength procedure to continue some of these nonlinear
normal modes. 

It is important to notice that the icosahedral symmetries appear also in adenoviruses with icosahedral capsid, or
other icosahedral molecules considered in \cite{We}. The methods presented
here are applicable to these cases as well.

\vskip.3cm

\paragraph{Contents.}

The rest of the paper is arranged as follows. In section
\ref{sec:fullerene} we present the model equations
appropriate for studying the dynamics of the fullerene molecule. In subsection
\ref{sec:coordinates}, we propose a new indexation for the fullerene atoms
which greatly simplifies the description of symmetries in the molecule. Then, in subsection \ref{sec:forcefield} we discuss the
choice of the force field for the fullerene molecule that seems to be the most
appropriate in order to model nonlinear vibrations and in subsection
\ref{sec:iso-symm} we describe the action of the group $I\times O(3)$ on the
space $\mathscr V$. Then, in subsection \ref{sec:minimizer}, we find the
minimizer of the potential $V$ among the configurations with icosahedral
symmetries by applying Palais criticality principle. In subsection
\ref{sec:iso-decomp}, we identify the $I$-isotypical decomposition of the
space $\mathscr V$ and use it to determine the spectrum of the operator
$\nabla^{2}V(u_{o})$ and the $I$-isotypical types of the corresponding
eigenspaces. In Section \ref{sec:eq-bif} we prove the bifurcation of periodic
solutions from the equilibrium configuration $u_{o}$ of the fullerene
molecule. In Section \ref{sec:symmetries} we describe the symmetries of the
periodic solutions.  In addition to the theoretical results stated Theorem \ref{th:main},
several of these symmetric periodic solutions were obtained 
by numerical continuation, for which the numerical data is shown graphically.
In Appendix \ref{sec:equi-degree}, we include
a short review of the gradient degree, including computational algorithms, and
the computations of the $I\times O(2)$-equivariant gradient degree.

\vskip.3cm


\section{Fullerene Model}

\label{sec:fullerene}

\subsection{Equations for Carbons}

\label{sec:coordinates}

In this section, we propose a new indexation for the fullerene atoms
which greatly simplifies the description of symmetries in the molecule: to
each atom we assigned two indices -- one (being a $5$-cycle in $S_{5}$)
indicating in which the side of the dodecahedron the atom is located and the
second indicating its position in that side (as it is illustrated on Figure
\ref{fig:fullerene})

Let $A_{5}$ be the alternating group of permutations of five elements
$\{1,2,3,4,5\}$. The five conjugacy classes in $A_{5}$ are listed in Table
\ref{tab:conjugacy}. The $C_{60}$ molecule is arranged in $12$ unconnected
pentagons of atoms. We implement the following notation for the indices of the
$60$ atoms (see Figure \ref{fig:fullerene}):

\begin{itemize}
\item $\tau\in\mathcal{C}_{4}$ is used to denote each of the $12$ pentagonal faces.

\item $k\in\{1,...,5\}=:\mathbb{Z}[1,5]$ is used to denote each of the $5$
vertices in the $12$ pentagonal faces.
\end{itemize}

\begin{table}[ptb]
\vskip.4cm
\par
\begin{center}%
\begin{tabular}
[c]{|c|c|c|c|c|}%
\noalign{\vskip2pt\hrule\vskip2pt} $\mathcal{C}_{1}$ & $\mathcal{C}_{2}$ &
$\mathcal{C}_{3}$ & $\mathcal{C}_{4}$ & $\mathcal{C}_{5}$\\
\noalign{\vskip2pt\hrule\vskip2pt} $(1)$ & $(12)(34)$, $(13)(24)$,$(14)(23)$ &
$(123)$, $(132)$ & $(12345)$ & $(12354)$\\
& $(12)(35)$, $(13)(25)$,$(15)(23)$ & $(124)$, $(142)$ & $(12453)$ &
$(12435)$\\
& $(12)(45)$, $(14)(25)$,$(15)(24)$ & $(125)$, $(152)$ & $(12534)$ &
$(12543)$\\
& $(13)(45)$, $(14)(35)$, $(15)(34)$ & $(134)$, $(143)$ & $(13254)$ &
$(13245)$\\
& $(23)(45)$, $(24)(35)$,$(25)(34)$ & $(135)$, $(153)$ & $(13542)$ &
$(13524)$\\
&  & $(145)$, $(154)$ & $(13425)$ & $(13452)$\\
&  & $(234)$, $(243)$ & $(14235)$ & $(14253)$\\
&  & $(235)$, $(253)$ & $(14352)$ & $(14325)$\\
&  & $(245)$, $(254)$ & $(14523)$ & $(14532)$\\
&  & $(345)$, $(354)$ & $(15243)$ & $(15234)$\\
&  &  & $(15432)$ & $(15423)$\\
&  &  & $(15324)$ & $(15342)$\\\hline
\end{tabular}
\end{center}
\caption{Conjugacy classes of elements in $A_{5}$.}%
\label{tab:conjugacy}%
\end{table}

We define the set of indices as
\[
\Lambda=\mathcal{C}_{4}\times\mathbb{Z}[1,5]~\text{.}%
\]
With these notations each index $(\tau,k)\in\Lambda$ represents a face and a
vertex in the face of the truncated icosahedron as it is shown on Figure
\ref{fig:fullerene}. The vectors that represent the positions of the carbon
atoms are%
\[
u_{\tau,k}\in\mathbb{R}^{3},
\]
and the vector for the $60$ positions is
\[
u=(u_{\tau,k})_{(\tau,k)\in\Lambda}\in\left(  \mathbb{R}^{3}\right)
^{60}\text{.}%
\]

The space $\left(  \mathbb{R}^{3}\right)  ^{60}$ is a representation of the
group $I\times O(3)$, where $I=A_{5}\times\mathbb{Z}_{2}$ stands for the full
icosahedral group. With this notation, the action of $I\times O(3)$ on $V$ has
a simple definition: the action of $\sigma\in A_{5}$ and $-1\in\mathbb{Z}_{2}$
in $u$ is defined in each component by%
\begin{equation}
\rho(\sigma)u_{\tau,k}=u_{\sigma^{-1}\tau\sigma,\sigma^{-1}(k)}~,\qquad
\rho(-1)u_{\tau,k}:=u_{\tau^{-1},k}~. \label{eq:act1}%
\end{equation}
And the action of the group $A\in O(3)$ is defined by
\[
Au=(Au_{\tau,k})_{(\tau,k)\in\Lambda}.
\]

\subsection{Force Field}

\label{sec:forcefield}

The system for the fullerene molecule is given by \eqref{eq:1}. To describe
the force field $\nabla V$ we recognize that:

\begin{itemize}
\item The $60$ edges in the $12$ pentagonal faces represent single bonds. For
these single bonds we define the function $S:\Lambda\rightarrow\Lambda,$%
\[
S(\tau,k)=(\tau,\tau(k))~, \quad\tau\in\mathcal{C}_{4},\; k\in\mathbb{Z}%
[1,5].
\]

\item The $30$ remaining edges in the hexagon, which connect the different
pentagonal faces, represent double bounds. For these double bonds we define
the function $D:\Lambda\rightarrow\Lambda$,%
\[
D(\tau,k)=(\sigma,k)\text{ with }\sigma=\left(  k,\tau^{2}(k),\tau(k),\tau
^{4}(k),\tau^{3}(k)\right)  .
\]

\end{itemize}

Using the above notation, the force field energy is elegantly expressed by
\[
V(u)=\sum_{\left(  \tau,k\right)  \in\Lambda}\left(  U(\left\vert u_{\tau
,k}-u_{S(\tau,k)}\right\vert )+\frac{1}{2}U(\left\vert u_{\tau,k}%
-u_{D(\tau,k)}\right\vert )+U_{\left(  \tau,k\right)  }(u)\right)  \text{,}%
\]
where the coefficient $\frac12$ before the second term is to eliminate the
double count bonds, and the term $U_{\left(  \tau,k\right)  }(u)$ includes the
bending and torsion forces.

Bond stretching is represent by potential%
\[
U(x)=E_{0}\left(  \left(  1-e^{-\beta(x-r_{0})}\right)  ^{2}-1\right)  ,
\]
where $r_{0}$ is the equilibrium bond length, $E_{0}$ is the bond energy and
$\beta^{-1}$ is the width of the energy. The term $U_{\left(  \tau,k\right)
}(u)$ includes bending and torsion forces given by%
\[
U_{\left(  \tau,k\right)  }(u)=E_{B}(\theta_{1})+E_{B}(\theta_{2}%
)+E_{B}(\theta_{3})+E_{T}(\phi_{1})+E_{T}(\phi_{2})+E_{T}(\phi_{3}),
\]
where the bending $E_{B}(\theta)$ around each atom in a molecule is governed
by the hybridization of orbitals and is given by%
\[
E_{B}(\theta)=\frac{1}{2}k_{0}(\cos\theta-\cos\theta_{0})^{2}=\frac{1}%
{2}k_{\theta}(\cos\theta+1/2)^{2},\quad\theta_{0}:=2\pi/3,
\]
(here $\theta_{o}$ is the equilibrium angle and $k_{0}$ is the bending force
constant), and the torsion energy $E_{T}(\phi)$ (which describes the energy
change associated with rotation around a bond with a four-atom sequence) is
given by%
\[
E_{T}(\phi)=\frac{1}{2}k_{\phi}\left(  1-\cos2\phi\right)  =k_{\phi}\left(
1-\cos^{2}\phi\right)  \text{.}%
\]
The torsion energy reaches a maximum value at angles $\phi=\pm\pi/2$.

Each carbon $\left(  \tau,k\right)  \in\Lambda$ has three angles,%
\begin{align*}
\cos\theta_{1}  &  =\frac{u_{\tau,k}-u_{S(\tau,k)}}{\left\vert u_{\tau
,k}-u_{S(\tau,k)}\right\vert }\bullet\frac{u_{\tau,k}-u_{S^{-1}(\tau,k)}%
}{\left\vert u_{\tau,k}-u_{S^{-1}(\tau,k)}\right\vert },\\
\cos\theta_{2}  &  =\frac{u_{\tau,k}-u_{D(\tau,k)}}{\left\vert u_{\tau
,k}-u_{D(\tau,k)}\right\vert }\bullet\frac{u_{\tau,k}-u_{S(\tau,k)}%
}{\left\vert u_{\tau,k}-u_{S(\tau,k)}\right\vert }~,\\
\cos\theta_{3}  &  =\frac{u_{\tau,k}-u_{D(\tau,k)}}{\left\vert u_{\tau
,k}-u_{D(\tau,k)}\right\vert }\bullet\frac{u_{\tau,k}-u_{S^{-1}(\tau,k)}%
}{\left\vert u_{\tau,k}-u_{S^{-1}(\tau,k)}\right\vert }.
\end{align*}
Clearly, the bond bending at each atom $\left(  \tau,k\right)  \in\Lambda$ is
$E_{B}(\theta_{1})+E_{B}(\theta_{2})+E_{B}(\theta_{3})$. Let
\[
n=\frac{u_{D(\tau,k)}-u_{S(\tau,k)}}{\left\vert u_{D(\tau,k)}-u_{S(\tau
,k)}\right\vert }\times\frac{u_{D(\tau,k)}-u_{S^{-1}(\tau,k)}}{\left\vert
u_{D(\tau,k)}-u_{S^{-1}(\tau,k)}\right\vert }%
\]
be the unit normal vector to the plane passing by $u_{D(\tau,k)}$,
$u_{S(\tau,k)}$ and $u_{S^{-1}(\tau,k)}$. Each carbon $\left(  \tau,k\right)
\in\Lambda$ has three torsion energies given by%
\begin{align*}
\cos\phi_{1}  &  =n\bullet n_{1}~,\qquad n_{1}=\frac{u_{(\tau,k)}%
-u_{S(\tau,k)}}{\left\vert u_{(\tau,k)}-u_{S(\tau,k)}\right\vert }\times
\frac{u_{(\tau,k)}-u_{S^{-1}(\tau,k)}}{\left\vert u_{(\tau,k)}-u_{S^{-1}%
(\tau,k)}\right\vert }~,\\
\cos\phi_{2}  &  =n\bullet n_{2}~,\qquad n_{2}=\frac{u_{(\tau,k)}%
-u_{D(\tau,k)}}{\left\vert u_{(\tau,k)}-u_{D(\tau,k)}\right\vert }\times
\frac{u_{(\tau,k)}-u_{S^{-1}(\tau,k)}}{\left\vert u_{(\tau,k)}-u_{S^{-1}%
(\tau,k)}\right\vert }~,\\
\cos\phi_{3}  &  =n\bullet n_{3}~,\qquad n_{3}=\frac{u_{(\tau,k)}%
-u_{D(\tau,k)}}{\left\vert u_{(\tau,k)}-u_{D(\tau,k)}\right\vert }\times
\frac{u_{(\tau,k)}-u_{S(\tau,k)}}{\left\vert u_{(\tau,k)}-u_{S(\tau
,k)}\right\vert }~.
\end{align*}
Then, the bond bending at each atom $\left(  \tau,k\right)  \in\Lambda$ is
$E_{T}(\phi_{1})+E_{T}(\phi_{2})+E_{T}(\phi_{3})$.

For this work, we use the parameters given in \cite{Be} , which are
$E_{0}=6.1322~eV$, $\beta=1.8502~A^{-1}$, $r_{0}=01.4322$~$A$, $k_{\theta
}=10~eV$, $k_{\phi}=0.346~eV$. In this paper we use exactly these values.

\subsection{Icosahedral Symmetries}

\label{sec:iso-symm} In order to make the system \eqref{eq:1} reference
point-depended, we define the subspace
\begin{equation}
\mathscr V:=\{u\in\left(  \mathbb{R}^{3}\right)  ^{60}:\sum_{(\sigma
,k)\in\Lambda}u_{\sigma,k}=0\} \label{eq:V}%
\end{equation}
and
\[
\Omega_{o}=\left\{  u\in\mathscr V:u_{\sigma,k}\not =u_{\tau,j}\right\}  .
\]
We have that $\mathscr V$ and $\Omega_{o}$ are flow-invariant for \eqref{eq:1}

By the properties of functions $S$ and $D$, the potential
\begin{equation}
V:\Omega_{o}\rightarrow\mathbb{R~} \label{eq:mol}%
\end{equation}
is well defined and $I$-invariant. Moreover, the potential $V$ is invariant by
rotations and reflections of the group $O(3)$ because bonding, bending and
torsion forces depend only on the norm of the distances among pairs of atoms.
Therefore, the potential $V$ is $I\times O(3)$-invariant,%
\[
V((\sigma,A)u)=V(u),\qquad(\sigma,A)\in I\times O(3)~.
\]

Let
\begin{equation}
\label{eq:infinitezimal}J_{1}:=\left[
\begin{array}
[c]{ccc}%
0 & 0 & 0\\
0 & 0 & -1\\
0 & 1 & 0
\end{array}
\right]  ,\quad J_{2}:=\left[
\begin{array}
[c]{ccc}%
0 & 0 & -1\\
0 & 0 & 0\\
1 & 0 & 0
\end{array}
\right]  ,\quad J_{3}:=\left[
\begin{array}
[c]{ccc}%
0 & -1 & 0\\
1 & 0 & 0\\
0 & 0 & 0
\end{array}
\right]
\end{equation}
be the three infinitesimal generators of rotations in $O(3)$, i.e., $e^{\phi
J_{1}}$, $e^{\theta J_{2}}$ and $e^{J_{3}\psi}$, where $\phi$, $\theta$ and
$\psi$ are the Euler angles. The angle between two adjacent pentagons in a
dodecahedron is $\theta_{0}=\arctan2$. Then, the rotation by $\pi$ that fixes
a pair of antipodal edges is%
\begin{equation}
A=e^{-\left(  \theta_{0}/2\right)  J_{2}}e^{\pi J_{3}}e^{\left(  \theta
_{0}/2\right)  J_{2}}=%
\begin{bmatrix}
-\frac{1}{\sqrt{5}} & 0 & \frac{2}{\sqrt{5}}\\
0 & -1 & 0\\
\frac{2}{\sqrt{5}} & 0 & \frac{1}{\sqrt{5}}%
\end{bmatrix}
\text{.}%
\end{equation}
The rotation by $2\pi/5$ of the upper pentagonal face of a dodecahedron is%
\begin{equation}
B=e^{\frac{2\pi}{5}J_{3}}=%
\begin{bmatrix}
\frac{-1+\sqrt{5}}{4} & -\sqrt{\frac{5+\sqrt{5}}{8}} & 0\\
\sqrt{\frac{5+\sqrt{5}}{8}} & \frac{-1+\sqrt{5}}{4} & 0\\
0 & 0 & 1
\end{bmatrix}
~\text{.}%
\end{equation}
The subgroup of $O(3)$ generated by the matrices $A$ and $B$ is isomorphic to
icosahedral group $A_{5}$. Indeed, the generators $A$ and $B$ satisfy the
relations%
\[
A^{2}=B^{5}=(AB)^{3}=\text{\textrm{Id\,}}\ .
\]
On the other hand, the group $A_{5}$ is generated by
\begin{equation}
a=(23)(45),\qquad b=\left(  12345\right)  \text{,}%
\end{equation}
and we have similar relations%
\[
a^{2}=b^{5}=(ab)^{3}=(1).
\]
Therefore, the explicit homomorphism $\rho:A_{5}\rightarrow\mathrm{SO}(3)$
defined on generators by $\rho(a):= A$ and $\rho( b):= B$ is the required
isomorphism $A_{5}\simeq\rho(A_{5})\subset SO(3)$. We extend
\[
\rho:A_{5}\times\mathbb{Z}_{2}\rightarrow O(3)
\]
with $\rho(-1)=-\text{\textrm{Id\,}}\in O(3)$, and consequently we obtain an
explicit identification of the full icosahedral group $I$ with $\rho(I)\subset
O(3)$.

\subsection{Icosahedral Minimizer}

\label{sec:minimizer}

Let $\tilde{I}$ be the icosahedral subgroup of $I\times O(3)$ given by
\[
\tilde{I}=\left\{  (\sigma,\rho(\sigma))\in I\times O(3):\sigma\in I\right\}
~.
\]
The fixed point space
\[
\Omega_{0}^{\tilde{I}}=\mathscr V^{\tilde{I}}\cap\Omega_{o}=\{\left(
a_{\tau,k}\right)  _{(\tau,k)\in\Lambda}\in\Omega_{0}:a_{\tau,k}=(\sigma
,\rho(\sigma))a_{\tau,k}\},
\]
consist of all the truncated icosahedral configurations. An equilibrium of the
fullerene molecule can be found as a minimizer of $V$ on these configurations.
More precisely, since $V$ is $I\times O(3)$-invariant, by Palais criticality
principle, the minimizer of the potential $V$ on the fixed-point space of
$\tilde{I}$ is a critical point of $V$.

To find the minimizer among configurations with symmetries $\tilde{I}$, we
parameterize the carbons positions by fixing the position of one of them. Let
\[
u_{b,1}=(x,0,z), \quad x,z\in\mathbb{R},
\]
then we have
\[
u_{b,1}=(\sigma,\rho(\sigma))u_{b,1}=\rho(\sigma)u_{\sigma^{-1}b\sigma
,\sigma^{-1}(1)}\text{~,}%
\]
and relations
\begin{equation}
u_{\sigma^{-1}b\sigma,\sigma^{-1}(1)}=\rho(\sigma)^{-1}u_{b,1}=\rho
(\sigma)^{-1}(x,0,z)^{T},\quad\sigma\in A_{5}, \label{rel}%
\end{equation}
allow us to determine the positions of all other coordinates of the vector
$u=\left(  u_{\tau,k}\right)  $.

Therefore, the representation of $u(x,z)$ given by (\ref{rel}) provides us a
parametrization of a connected component of $\Omega_{0}^{\tilde{I}}/O(3)$ with
the domain
\[
\mathcal{D}=\{(x,z)\in\mathbb{R}^{2}:0<z,\;x<Cz\}
\]
where $C>0$ is a number determined by the geometric restrictions. We define
$v:\mathcal{D}\rightarrow\mathbb{R}$ by
\[
v(x,z)=V(u(x,z)).
\]
Since $v(x,z)$ is exactly the restriction of $V$ to the fixed-point subspace
$\Omega_{0}^{\tilde{I}}/O(3)$, then by the symmetric criticality condition, a
critical point of $v:\mathcal{D}\rightarrow\mathbb{R}$ is also a critical
point of $V$.

We implemented a numerical minimizing procedure to find the minimizer
$(x_{o},z_{o})$ of $v(x,z)$. We denote the truncated icosahedral minimizer
corresponding to the fullerene $C_{60}$ as
\[
u_{o}=u\left(  x_{o},z_{o}\right)  \in\mathscr V.
\]
The lengths of single and double bonds for this minimizer are given by%
\begin{align*}
d_{S}=\left\vert u_{b,1}-u_{S(b,1)}\right\vert =1.438084,\\
d_{D}=\left\vert u_{b,1}-u_{D(b,1)}\right\vert =1.420845\text{,}%
\end{align*}
respectively. These results are in accordance with the distances measured in
the paper \cite{He}.

An advantage of the notation $u=\left(  u_{\tau,k}\right)  $ is that it is
easy to visualize the elements associated to the rotations $\rho(\sigma)$ in
the truncated icosahedron (Figure \ref{fig:fullerene}). In these
configurations we have $\rho(\sigma)u_{\tau,k}=\sigma^{-1}u_{\tau,k}%
=u_{\sigma\tau\sigma^{-1},\sigma(k)}$, then $\rho(\sigma)$ is identified with
the rotation that sends the face $\tau$ to $\sigma\tau\sigma^{-1}$ and the
carbon atom identified by $k$ to $\sigma(k)$; for example, under the $\pi
$-rotation given by $\rho(a)=A$, face $b=(12345)$ goes to the face
$aba^{-1}=(13254)$, and the element $k=1$ to $a(1)=1$. In this manner, we
conclude that the elements of the conjugacy classes $\mathcal{C}_{2}$,
$\mathcal{C}_{3}$, $\mathcal{C}_{4}$ and $\mathcal{C}_{5}$ correspond to the
$15$ rotations by $\pi$, the $20$ rotations by $2\pi/3$, the $12$ rotations by
$2\pi/5$ and the $12$ rotations by $\pi/5$, respectively.

\subsection{Isotypical Decomposition and Spectrum of $\nabla^{2}V(u_{o})$}

\label{sec:iso-decomp}

The space $\mathscr V$ is an orthogonal complement in $\widetilde
{\mathscr V}=(\mathbb{R}^{3})^{60}$ of the subspace $\{(v,v,\dots,v):
v\in\mathbb{R}^{3}\}$, thus it is $\tilde{I}$-invariant. Given that the system
\eqref{eq:1}, $u(t)\in\mathscr V$, is symmetric with respect the group action
of $I\times O(3)$, the orbit of equilibria $u_{o}$ is a $3$-dimensional
submanifold in $\mathscr V$. To describe the tangent space, we define
\[
\mathcal{J}_{j}u=(J_{j}u_{\sigma,k}),
\]
where $J_{j}$ are the three infinitesimal generators of the rotations defined
in \eqref{eq:infinitezimal}. Then, the slice $S_{o}$ to the orbit of $u_{o}$
is
\[
S_{o}:=\{x\in\mathscr V:x\bullet\mathcal{J}_{j}u_{0}=0,\quad j=1,2,3\}.
\]
Since $u_{o}$ has the isotropy group $\tilde{I}$, then $S_{o}$ is an
orthogonal $\tilde{I}$ representation.

In this section we identify the $\tilde{I}$-isotypical components of $S_{o}$.
In order to simplify the notation, hereafter we identify $\tilde{I}$ with the
group $A_{5}\times\mathbb{Z}_{2}$,%
\[
\tilde{I}=A_{5}\times\mathbb{Z}_{2}~.
\]
Put $\varphi_{\pm}=\frac{1}{2}(1\pm\sqrt{5})$, and consider the permutations
$a=(2,3)(4,5)$, $b=\left(  1,2,3,4,5\right)  $ and $c=(1,2,3)$. The character
table of $A_{5}$ is:%

\[%
\begin{tabular}
[c]{|cc|ccccc|}\hline
Rep. & Character & $(1)$ & $a$ & $c$ & $b$ & $b^{2}$\\\hline
$\mathcal{V}_{1}$ & $\chi_{1}$ & $1$ & $1$ & $1$ & $1$ & $1$\\
$\mathcal{V}_{2}$ & $\chi_{2}$ & $4$ & $0$ & $1$ & $-1$ & $-1$\\
$\mathcal{V}_{3}$ & $\chi_{3}$ & $5$ & $1$ & $-1$ & $0$ & $0$\\
$\mathcal{V}_{4}$ & $\chi_{4}$ & $3$ & $-1$ & $0$ & $\varphi_{+}$ &
$\varphi_{-}$\\
$\mathcal{V}_{5}$ & $\chi_{5}$ & $3$ & $-1$ & $0$ & $\varphi_{-}$ &
$\varphi_{+}$\\\hline
\end{tabular}
\ \ \ \ \
\]
The character table of $\tilde{I}\simeq A_{5}\times\mathbb{Z}_{2}$ is obtained
from the table of $A_{5}$. We denote the irreducible representations of $I$
by\ $\mathcal{V}_{\pm n}$ for $n=1,...,5$, where the element $-1\in
\mathbb{Z}_{2}$ acts as $\pm\text{\textrm{Id\,}}$ in $\mathcal{V}_{\pm n}$ and
elements $\gamma\in A_{5}$ act as they act on $\mathcal{V}_{n}$. Notice that
all the representations $\mathcal{V}_{\pm n}$ are absolutely irreducible.
Therefore, the character table for $A_{5}\times\mathbb{Z}_{2}$ is as follows:
\begin{table}[ptb]
\begin{center}%
\begin{tabular}
[c]{|c|cccccccccc|}\hline
& (1) & $a$ & $c$ & $b$ & $b^{2}$ & $(-1)$ & $(a,-1) $ & $(c, -1) $ & $(b,-1)$
& $(b^{2},-1) $\\\hline
$\chi_{1}$ & 1 & 1 & 1 & 1 & 1 & 1 & 1 & 1 & 1 & 1\\
$\chi_{2}$ & 4 & 0 & 1 & -1 & -1 & 4 & 0 & 1 & -1 & -1\\
$\chi_{3}$ & 5 & 1 & -1 & 0 & 0 & 5 & 1 & -1 & 0 & 0\\
$\chi_{4}$ & 3 & -1 & 0 & $\varphi_{+} $ & $\varphi_{-}$ & 3 & -1 & 0 &
$\varphi_{+}$ & $\varphi_{-}$\\
$\chi_{5}$ & 3 & -1 & 0 & $\varphi_{-}$ & $\varphi_{+}$ & 3 & -1 & 0 &
$\varphi_{-}$ & $\varphi_{+}$\\
$\chi_{-1}$ & 1 & 1 & 1 & 1 & 1 & -1 & -1 & -1 & -1 & -1\\
$\chi_{-2}$ & 4 & 0 & 1 & -1 & -1 & -4 & 0 & -1 & 1 & 1\\
$\chi_{-3}$ & 5 & 1 & -1 & 0 & 0 & -5 & -1 & 1 & 0 & 0\\
$\chi_{-4}$ & 3 & -1 & 0 & $\varphi_{+} $ & $\varphi_{-}$ & -3 & 1 & 0 &
$-\varphi_{+} $ & $-\varphi_{-}$\\
$\chi_{-5}$ & 3 & -1 & 0 & $\varphi_{-}$ & $\varphi_{+} $ & -3 & 1 & 0 &
$-\varphi_{-} $ & $-\varphi_{+}$\\\hline
\end{tabular}
\end{center}
\caption{Character table for $A_{5}\times\mathbb{Z}_{2}$}%
\label{tab:I}%
\end{table}\vskip.3cm By comparing the character $\chi_{_{\mathscr V}}$ withe
the characters in Table \ref{tab:I}, we obtain the following $I$-isotypical
decomposition of $\mathscr V$
\[
\mathscr V=\bigoplus_{n=1}^{5} \mathscr V_{n}\oplus\mathscr V_{-n},
\]
where $\mathscr V_{\pm n}$ is modeled on $\mathcal{V}_{\pm n}$.

\vskip.3cm

We numerically computed the spectrum $\{\mu_{j}: j=1,2,\dots, 47\}$ of the
Hessian $\nabla^{2}V(u_{o})$ at this minimizer $u_{o}$. Since $\nabla
^{2}V(u_{o}):\mathscr V\rightarrow\mathscr V$ is $\tilde{I}$-equivariant,
\[
\nabla^{2}V(u_{o})|_{\mathscr{V}_{n}\cap E(\mu_{j})}=\mu_{j}%
\,\text{\textrm{Id}}:\mathscr{V}_{n}\cap E(\mu_{j})\rightarrow\mathscr{V}_{n}%
\cap E(\mu_{j}),
\]
where $E(\mu_{j})$ stands for the eigenspace corresponding to $\mu_{j}$. We
found that each of the eigenspaces $E(\mu_{j})$ is an irreducible
subrepresentation of $\mathscr V$ , i.e. the isotypical multiplicity of
$\mu_{j}$ is simple. Including the zero eigenspace, we have $47$ different
components. Thus we have that $\sigma(\nabla^{2}V(u_{o})|_{{S_{o}}})=\{\mu
_{1},...,\mu_{46}\}$ with $\mu_{j}>0$, so
\begin{equation}
S_{o}=\bigoplus_{j=1}^{46}\mathcal{V}_{n_{j}} \quad\text{ and } \quad
\nabla^{2}V(u_{o})|_{\mathcal{V}_{n_{j}}}=\mu_{j}\,\text{\textrm{Id}.}
\label{eq:S4-iso-S}%
\end{equation}

In order to determine in which isotypical component $\mathscr V_{n_{j}}$ the
eigenspace $E(\mu_{j})$ is contained for a given eigenvalue $\mu_{j}$, we
apply the isotypical projections
\[
P_{{n}}v:=\frac{\dim(\mathcal{V}_{n})}{120}\sum_{g\in\tilde{I}}\chi
_{n}(g)~gv,\quad v\in\mathscr V,
\]
where $\mathcal{V}_{n}$, $n=\pm1,\dots,\pm5$, is the irreducible
representation identified by the character Table \ref{tab:I}. Then the
component $\mathscr V_{n_{j}}$ can be clearly identified by the projection
$P_{n_{j}}$.

\begin{table}[t]
\begin{center}
\scalebox{.7}{
\begin{tabular}
[c]{|c|cccc|}\hline
$j$ & \text{{\small Mult.}} & $\mu_{j}$ & $\lambda_j$&${\small n}_{j}$\\
\hline
{\small 1} & {\small 5} & {\small 176.536} &{\small $ 0.075263 $}& -{\small 3}\\
{\small 2} & {\small 5} & {\small 176.366} &{\small $ 0.075300 $}& {\small 3}\\
{\small 3} & {\small 4} & {\small 164.083} &{\small $ 0.078067 $}& {\small 2}\\
{\small 4} & {\small 4} & {\small 160.292} &{\small $ 0.078985 $}& -{\small 2}\\
{\small 5} & {\small 3} & {\small 159.290} &{\small $ 0.079233$}& -{\small 5}\\
{\small 6} & {\small 5} & {\small 148.597} &{\small $ 0.082034 $}& {\small 3}\\
{\small 7} & {\small 3} & {\small 141.071} &{\small $ 0.084194 $}& -{\small 4}\\
{\small 8} & {\small 3} & {\small 140.573} &{\small $ 0.084343 $}& {\small 5}\\
{\small 9} & {\small 1} & {\small 135.632} &{\small $ 0.085866 $}& {\small 1}\\
{\small 10} & {\small 4} & {\small 134.935} &{\small $ 0.086087$}& -{\small 2}\\
{\small 11} & {\small 4} & {\small 129.544} &{\small $ 0.087860 $}& {\small 2}\\
{\small 12} & {\small 5} & {\small 125.431} &{\small $ 0.089289 $}& -{\small 3}\\
{\small 13} & {\small 3} & {\small 107.719} &{\small $ 0.096350 $}& {\small 4}\\
{\small 14} & {\small 5} & {\small 98.5525} &{\small $ 0.100732 $}& {\small 3}\\
{\small 15} & {\small 3} & {\small 93.4648} &{\small $ 0.103437 $}& -{\small 5}\\
{\small 16} & {\small 5} & {\small 87.7541} &{\small $0.106750  $}& -{\small 3}\\
{\small 17} & {\small 3} & {\small 83.9718} &{\small $ 0.109127 $}& -{\small 4}\\
{\small 18} & {\small 4} & {\small 71.6288} &{\small $ 0.118156 $} &{\small 2}\\
{\small 19} & {\small 5} & {\small 67.1181} &{\small $ 0.122062 $} &{\small 3}\\
{\small 20} & {\small 1} & {\small 59.3865} &{\small $  0.129765$}& -{\small 1}\\
{\small 21} & {\small 3} & {\small 50.4797} &{\small $0.140748 $}& -{\small 5}\\
{\small 22} & {\small 4} & {\small 47.5646} &{\small $  0.144997$}& -{\small 2}\\
{\small 23} & {\small 3} & {\small 42.2947} &{\small $ 0.153765$}& {\small 4}\\
\hline
\end{tabular}
\qquad\begin{tabular}
[c]{|c|cccc|}\hline
$j$ & Mult. & ${\small \mu}_{j}$ &  $\lambda_j$&$n_{j}$\\
\hline
{\small 24} & {\small 3} & {\small 41.3918} &{\small $ 0.155433$}& {\small 5}\\
{\small 25} & {\small 4} & {\small 33.9885} &{\small $  0.171528$}& -{\small 2}\\
{\small 26} & {\small 5} & {\small 28.8031} &{\small $ 0.186329 $}& -{\small 3}\\
{\small 27} & {\small 5} & {\small 27.4795} &{\small $ 0.190764 $}& {\small 3}\\
{\small 28} & {\small 3} & {\small 27.3153} &{\small $ 0.191336 $}& {\small 5}\\
{\small 29} & {\small 4} & {\small 25.5388} &{\small $ 0.197879 $}& -{\small 2}\\
{\small 30} & {\small 4} & {\small 22.7212} &{\small $  0.209790$}& {\small 2}\\
{\small 31} & {\small 5} & {\small 19.4536} &{\small $ 0.226725 $}& -{\small 3}\\
{\small 32} & {\small 5} & {\small 19.3377} & {\small $ 0.227404 $}&{\small 3}\\
{\small 33} & {\small 3} & {\small 19.2379} &{\small $ 0.227993 $}& -{\small 5}\\
{\small 34} & {\small 4} & {\small 16.5356} &{\small $ 0.245918 $} &{\small 2}\\
{\small 35} & {\small 3} & {\small 16.5255} & {\small $ 0.245993 $}&{\small 5}\\
{\small 36} & {\small 3} & {\small 15.1033} & {\small $ 0.257314 $}&-{\small 4}\\
{\small 37} & {\small 3} & {\small 10.3908} & {\small $0.310224  $}&{\small 4}\\
{\small 38} & {\small 1} & {\small 10.2520} &{\small $ 0.312317 $} &{\small 1}\\
{\small 39} & {\small 5} & {\small 10.1098} &{\small $ 0.314506 $} &-{\small 3}\\
{\small 40} & {\small 3} & {\small 9.03077} &{\small $ 0.332765 $} &-{\small 4}\\
{\small 41} & {\small 4} & {\small 9.02666} &{\small $ 0.332841 $}& {\small 2}\\
{\small 42} & {\small 5} & {\small 6.99929} & {\small $ 0.377984 $}&{\small 3}\\
{\small 43} & {\small 5} & {\small 6.95354} & {\small $  0.379225$}&-{\small 3}\\
{\small 44} & {\small 3} & {\small 5.42311} &{\small $0.429414  $} &-{\small 5}\\
{\small 45} & {\small 4} & {\small 5.26429} & {\small $ 0.435843$}&-{\small 2}\\
{\small 46} & {\small 5} & {\small 3.04384} &{\small $  0.573177$} &{\small 3}\\
\hline
\end{tabular}}
\end{center}
\caption{Eigenvalues $\mu_{j}$ of $\nabla^{2}V(u_{o})$ and critical numbers
$\lambda_{j}$ according to their isotypical type $\mathcal{V}_{n_{j}}$.}\label{tab:crit-values}
\end{table}

In the following Table \ref{tab:crit-values}, we show the number $n_{j}%
\in\{-5,...,-1,1,...,5\}$ that identifies the irreducible representation
corresponding to the eigenvalue $\mu_{j}$ for $j=1,..,46$. The numerical
computations strongly indicate that all the eigenvalues $\mu_{j}$,
$j=1,..,46$, are non-resonant. 

\begin{remark}
\textrm{The models proposed in \cite{Wa} and \cite{Be}, consider the presence
of van der Waals forces among carbon atoms, which are modeled by the
potential
\[
W(x)=\varepsilon\left(  \frac{\sigma^{12}}{x^{12}}-2\frac{\sigma^{6}}{x^{6}%
}\right)  ~,
\]
where $\sigma=3.4681~A^{-1}$ is the minimum energy distance and $\epsilon
=0.0115~eV$ the depth of this minimum. Our numerical computations indicate
that the models with van der Waals forces do not produce acceptable bond
lengths between the atoms (as it is given in \cite{He}), neither the spectrums
fit the experimental data (cf. \cite{Ho}). Actually, the models \cite{Wa} and
\cite{Be} without van der Waals forces lead to results which correctly
approximate the measurements in \cite{He} and \cite{Ho} (for bond lengths
$d_{S}$ and $d_{D}$ and frequencies $\sqrt{\mu_{j}/m}$, which are within the
range $100$ to $1800$ $cm^{-1}$). }
\end{remark}

\section{Equivariant Bifurcation}

\label{sec:eq-bif}

In what follows, we are interested in finding non-trivial $T$-periodic
solutions to \eqref{eq:1}, bifurcating from the $G$-orbit of the equilibrium point
$u_{o}$. By normalizing the period, i.e. by making the substitution
$v(t):=u\left(  \lambda t\right)  $ in \eqref{eq:1}, we obtain the system
\begin{equation}%
\begin{cases}
\ddot{v}=-\lambda^{2}\nabla V(v),\\
v(0)=v(2\pi),\;\;\dot{v}(0)=\dot{v}(2\pi),
\end{cases}
\label{eq:mol1}%
\end{equation}
where $\lambda^{-1}=2\pi/T$ is the frequency.

\subsection{Equivariant Gradient Map}

Since $\mathscr V$ is an orthogonal $I\times O(3)$- representation, we can
consider the first Sobolev space of $2\pi$-periodic functions from
$\mathbb{R}$ to $\mathscr V$, i.e.,
\[
H_{2\pi}^{1}(\mathbb{R},\mathscr V):=\{u:\mathbb{R}\rightarrow\mathscr
V\;:\;u(0)=u(2\pi),\;u|_{[0,2\pi]}\in H^{1}([0,2\pi];\mathscr V)\},
\]
equipped with the inner product
\[
\langle u,v\rangle:=\int_{0}^{2\pi}(\dot{u}(t)\bullet\dot{v}(t)+u(t)\bullet
v(t))dt~.
\]

Let $O(2)=SO(2)\cup\kappa SO(2)$ denote the group of $2\times2$-orthogonal
matrices, where
\[
\kappa=\left[
\begin{array}
[c]{cc}%
1 & 0\\
0 & -1
\end{array}
\right]  ,\qquad\left[
\begin{array}
[c]{cc}%
\cos\tau & -\sin\tau\\
\sin\tau & \cos\tau
\end{array}
\right]  \in SO(2)~.
\]
It is convenient to identify a rotation with $e^{i\tau}\in S^{1}%
\subset\mathbb{C}$. Notice that $\kappa e^{i\tau}=e^{-i\tau}\kappa$, i.e.,
$\kappa$ as a linear transformation of $\mathbb{C}$ into itself, acts as
complex conjugation.

Clearly, the space $H_{2\pi}^{1}(\mathbb{R},\mathscr V)$ is an orthogonal
Hilbert representation of
\[
G:=I\times O(3)\times O(2).
\]
Indeed, we have for $u\in H_{2\pi}^{1}(\mathbb{R},\mathscr V)$ and
$(\sigma,A)\in I\times O(3)$ (see \eqref{eq:act1})
\begin{align}
\left(  \sigma,A\right)  u(t)  &  =(\sigma,A)u(t),\label{eq:ac}\\
e^{i\tau}u(t)  &  =u(t+\tau),\nonumber\\
\kappa u(t)  &  =u(-t).\nonumber
\end{align}

It is useful to identify a $2\pi$-periodic function $u:\mathbb{R}\rightarrow
V$ with a function $\widetilde{u}:S^{1}\rightarrow\mathscr V$ via the map
{$\mathfrak{e}(\tau)=e^{i\tau}:\mathbb{R}$}$\rightarrow S^{1}$. Using this
identification, we will write $H^{1}(S^{1},\mathscr V)$ instead of $H_{2\pi
}^{1}(\mathbb{R},\mathscr V)$. Let
\[
\Omega:=\{u\in H^{1}(S^{1},\mathscr V):u(t)\in\Omega_{o}\}.
\]
We define $J:\mathbb{R}\times\Omega\rightarrow\mathbb{R}$ by
\begin{equation}
J(\lambda,u):=\int_{0}^{2\pi}\left[  \frac{1}{2}|\dot{u}(t)|^{2}-\lambda
^{2}V(u(t))\right]  dt. \label{eq:var-1}%
\end{equation}
Then, the system \eqref{eq:mol1} can be written as the following variational
equation
\begin{equation}
\nabla_{u}J(\lambda,u)=0,\quad(\lambda,u)\in\mathbb{R}\times\Omega.
\label{eq:bif1}%
\end{equation}

Consider $u_{o}\in\mathscr V$ -- the equilibrium point of \eqref{eq:mol} (i.e.
symmetric ground state) described in previous section. Then $u_{o}$ is a
critical point of $J$. We are interested in finding non-stationary $2\pi
$-periodic solutions bifurcating from $u_{o}$, i.e. non-constant solutions to
system \eqref{eq:bif1}. We consider the orbit $G(u_{o})$ of $u_{o}$ in
$H^{1}(S^{1},\mathscr V)$. We denote by $\mathcal{S}_{o}$ the slice to
$G(u_{o})$ in $H^{1}(S^{1},\mathscr V)$. We consider the $G_{u_{o}}$-invariant
restriction $J:\mathbb{R}\times\left(  \mathcal{S}_{o}\cap\Omega\right)
\rightarrow\mathbb{R}$ of $J$ to the set $\mathcal{S}_{o}\cap\Omega$. This
restriction will allow us to apply the Slice Criticality Principle (see
Theorem \ref{thm:SCP}) in order to compute the gradient equivariant degree of
$\nabla J_{\lambda}$ on a small neighborhood $\mathscr U$ of $G(u_{o})$
needed for evaluation of the equivariant invariant $\omega_{G}(\lambda)$.

Consider the operator $L:H^{2}(S^{1};\mathscr V)\rightarrow L^{2}%
(S^{1};\mathscr V)$, given by
\[
Lu=-\ddot{u}+u
\]
for $u\in H^{2}(S^{1},\mathscr
V)$. Then the inverse operator $L^{-1}$ exists and is bounded. Let
$j:H^{2}(S^{1};\mathscr V)\rightarrow H^{1}(S^{1},\mathscr V)$ be the natural
embedding operator. Clearly, $j$ is a compact operator. Then, one can easily
verify that
\begin{equation}
\nabla_{u}J(\lambda,u)=u-j\circ L^{-1}(\lambda^{2}\nabla V(u)+u),
\label{eq:gradJ}%
\end{equation}
where $u\in H^{1}(S^{1},\mathscr V)$. Consequently, the bifurcation problem
\eqref{eq:bif1} is equivalent to $\nabla_{u}J(\lambda,u)=0$. Moreover, we
have
\begin{equation}
\nabla_{u}^{2}J(\lambda,u_{o})v=v-j\circ L^{-1}(\lambda^{2}\nabla^{2}%
V(u_{o})v+v)~, \label{eq:D2J}%
\end{equation}
where $v\in H^{1}(S^{1},\mathscr V)$.

Notice that
\[
\mathscr A(\lambda):=\nabla_{u}^{2}J(\lambda,u_{o})|_{\mathcal{S}_{o}%
}:\mathcal{S}_{o}\rightarrow\mathcal{S}_{o}.
\]
Thus, by implicit function theorem, $G(u_{o})$ is an isolated orbit of
critical points, whenever $\mathscr A(\lambda)$ is an isomorphism. Therefore,
if a point $(\lambda_{o},u_{o})$ is a bifurcation point for \eqref{eq:bif1},
then $\mathscr A(\lambda_{o})$ cannot be an isomorphism. In such case we
define
\[
\Lambda:=\{\lambda>0:\mathscr A(\lambda_{o})\text{ is not an isomorphism}\}~,
\]
and call this set the \textit{critical set} for the trivial solution $u_{o}$.

\subsection{Critical Numbers}

Consider the $S^{1}$-action on $H^{1}(S^{1},\mathscr V)$, where $S^{1}$ acts
on functions by shifting the argument (see \eqref{eq:ac}). Then, $(H^{1}%
(S^{1},\mathscr V))^{S^{1}}$ is the space of constant functions, which can be
identified with the space $\mathscr V$, i.e.,
\[
H^{1}(S^{1},\mathscr V)=\mathscr V\oplus\mathscr W,\quad\mathscr W:=\mathscr
V^{\perp}.
\]
Then, the slice $\mathcal{S}_{o}$ in $H^{1}(S^{1},\mathscr V)$ to the orbit
$G(u_{o})$ at $u_{o}$ is exactly
\[
\mathcal{S}_{o}=S_{o}\oplus\mathscr W.
\]
Consider the $S^{1}$-isotypical decomposition of $\mathscr W$, i.e.,
\[
\mathscr W=\overline{\bigoplus_{l=1}^{\infty}\mathscr W_{l}},\quad\mathscr
W_{l}:=\{\cos(l\cdot)\mathfrak{a}+\sin(l\cdot)\mathfrak{b}:\mathfrak{a}%
,\,\mathfrak{b}\in\mathscr V\}
\]
In a standard way, the space $\mathscr W_{l}$, $l=1,2,\dots$, can be naturally
identified with the complexification  $\mathscr V^{\mathbb{C}}$ on which $S^{1}$ acts by
$l$-folding,
\[
\mathscr W_{l}=\{e^{il\cdot}z:z\in\mathscr V^{\mathbb{C}}\}.
\]

Since the operator $\mathscr A(\lambda)$ is $G_{u_{o}}$-equivariant with
\[
G_{u_{o}}=\tilde{I}\times O(2),
\]
it is also $S^{1}$-equivariant and thus $\mathscr A(\lambda)(\mathscr
W_{l})\subset\mathscr W_{l}$. Using the $\tilde{I}$-isotypical decomposition
of $\mathscr V^{\mathbb{C}}$, we have the $G_{u_{o}}$-invariant decomposition
\[
\mathscr W_{l}=\bigoplus_{j=1}^{46}\mathcal{V}_{{j},l},\quad  \mathcal{V}_{{j},l}:=\{e^{il\cdot}z:z\in E(\mu_j)^\bc \},
\]
where $\mathcal{V}_{n_{j},l}=\mathcal V_{n_j}^\bc=\bc\otimes_\br \mathcal V_{n_j}$ is the $I\times O(2)$-irreducible representation with  $O(2)$ acting on $\bc$ by $l$-folding and complex conjugation.
 We have
\[
\mathscr A(\lambda)|_{\mathcal{V}_{{j},l}}=\left(  1-\frac{\lambda^{2}%
\mu_{j}+1}{l^{2}+1}\right)  \id.
\]
Thus $A(\lambda_{o})|_{\mathcal{V}_{{j},l}}=0$ if and only if
$\lambda_{o}^{2}=l^{2}/\mu_{j}$ for some $l=1,2,3,\dots$ and $j=0,1,2, \dots, 46$.

We  will write 
\[
\lambda_{j,l}=\frac{l}{\sqrt{\mu_{j}}}
\]
to denote the critical numbers in $\Lambda$. Then the critical set for the equilibrium $u_{o}$ of system \eqref{eq:mol}
is
\[
\Lambda= \{ \lambda_{j,l}:j=0,...,46,\quad l=1,2,3,\dots \} .
\]

Let us point out that in the case of isotypical resonances,  the critical numbers may not be  uniquely identified by indices $(j,l)$. The first and last critical numbers for $l=1$ are
$\lambda_{1,1}=.07526\,3$ and $\lambda_{46,1}=0.573\,18$, respectively. We
computed numerically (with precision $10^{-5}$) all the different values $\lambda_{j,l}$ from $\lambda_{1,1}$ to
$\lambda_{46,1}$. We obtain that among these approximations there is no-resonance
with harmonic critical number from $\lambda_{1,1}$ to $\lambda_{46,1}$, i.e.,%
\begin{equation}\label{eq:reson-no}
\lambda_{1,1}<\lambda_{2,1}<\lambda_{3,1}<\lambda_{4,1}%
<...<\lambda_{5,7}<\ \lambda_{26,3}<\lambda_{21,4}<\lambda_{27,3}%
<\lambda_{46,1}. 
\end{equation}
Therefore, a plausible assumption under the numerical evidence is that all the eigenvalues $\lambda_{j,1}$ are isotypical non-resonant for  $j=1,...,46$.

\subsection{Conjugacy Classes of Subgroups in $I\times O(2)$}\label{sec:conjugacy}

In order to simplify the notation, in what follows, instead of using the symbol $\tilde{I}$, we will write $I$. Under this notation the isotropy group $G_{u_o}$ is
\[
G_{u_{o}}=I\times O(2).
\]
The notation in this section is useful to obtain the classification of all
conjugacy classes $(\mathscr H)$ of closed subgroups in $I\times O(2)$.

The  representatives of the conjugacy classes of the subgroups in $A_5\times \bz_2$ consisting of proper nontrivial subgroups of $A_5$ are:
\begin{align*}
\bz_2&=\{ ((1),1),\, ((12)(34),1)\},\\
\bz_3&=\{ ((1),1)\,, ((123),1),\, ((132),1)\},\\
V_4&=\{((1),1),\, ((12)(34),1),\,  ((13)(24),1),\, ((23)(14),1)\},\\
\bz_5&=\{((1),1), \,((12345),1),\, ((13524),1),\,  ((14253),1),\, ((15432),1)\},\\
D_3&=\{((1),1),\, ((123),1),\, ((132),1),\, ((12)(45),1),\, ((13)(45),1),\, ((23)(45),1)\},\\
A_4&=\{((1),1),\, ((12)(34),1),\, ((13)(24),1),\, ((14)(23),1),\,((123),1),\,  ((132),1),\,   ( (124),1),\\
 &\hskip.5cm ((142),1,\,  ((134),1),\,   ((143),1),\, ( (234),1),\, ((243),1)\},\\
D_5&=\{((1),1),\, ((12345),1),\, ((13524),1),\,  ((14253),1),\, ((15432),1),\, ((12)(35),1),\, ((13)(45),1),\\
&\hskip.5cm ((14)(23),1),
 \,( (15)(24),1),\,  ((25)(34)\}.
\end{align*}

The  representatives of the additional conjugacy classes of the subgroups in $A_5\times \bz_2$ will be used to describe the symmetries of nonlinear vibrations. Besides of the product subgroups  $H^{p}:=H\times\mathbb{Z}_{2}$, we have the also following  twisted subgroups $H^\varphi$ of $A_5\times \bz_2$ (where $H$ is a subgroup of $A_5$):
{\small\begin{align*}
\bz_2^z&=\Big\{\big((1),1\big),\, \big((12)(34),-1\big)\Big\},\\
V_4^z&=\Big\{\big((1),1\big),\, \big((12)(34),-1\big),\, \big((13)(24),-1\big),\,  \big((23)(14),1\big)\Big\},\\
D_3^z&=\Big\{\big((1),1\big),\, \big((123),1\big),\, \big((132),1\big),\,\big((12)(45),-1\big),\, \big((13)(45),-1\big),\\
&\hskip.5cm \, \big((23)(45),-1\big)\Big\},\\
D_5^z&=\Big\{\big((1),1\big),\,  \big((12345),1\big),\, \big((13524),1\big),\, \big(14253),1\big),\,  \big((15432),1\big),\\
&\hskip.5cm \, \big((12)(35),-1\big),\, \big((13)(45),-1\big),\, \big((14)(23),-1\big),\,  \big((15)(24),-1\big),\\
&\hskip.5cm \, \big((25)(34),-1\big)\Big\}.
\end{align*}} With these definitions the subconjugacy lattice for $A_5\times \bz_2$ is shown in Figure \ref{fig:sub-A5}.

\begin{figure}
\vglue-4cm
\begin{center}
\includegraphics[height=15cm]{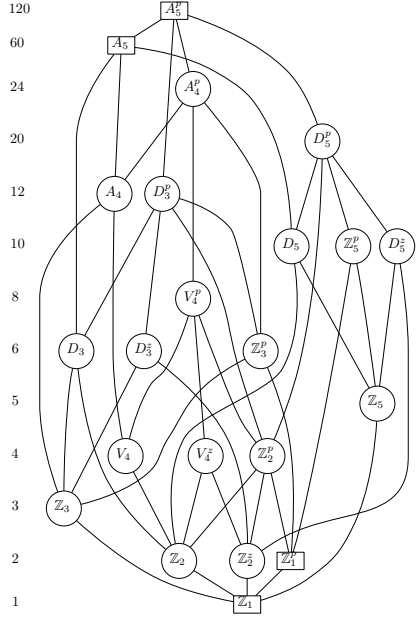}
\end{center}
\vskip-3.4cm
\caption{Lattice of conjugacy classes of subgroups in $A_5\times \bz_2$. The square boxes indicate that the related subgroup is normal in $A_5\times \bz_2$.}\label{fig:sub-A5}
\end{figure}

The result (see \cite{DaKr,Goursat}) provides a description
of subgroups of the product group $I\times O(2)$. Namely, for any subgroup $\mathscr
H$ of the product group $I\times O(2)$, there exist subgroups $H\leq I$  and $K\leq O(2)$, a group $L$ and two epimorphisms $\varphi
:H\rightarrow L$ and $\psi:K\rightarrow L$ such that
\[\mathscr H=\{(h,k)\in H\times K:\varphi(h)=\psi(k)\}.\]
In order to make the notation self-contained, we will
assume that $L=K/\ker(\psi)$, so $\psi:K\rightarrow L$ is evidently the
natural projection. On the other hand, the group $L$ can be naturally
identified with a finite subgroup of $O(2)$ being either $D_{n}$ or
$\mathbb{Z}_{n}$. Since we are interested in describing conjugacy
classes of $\mathscr H$, we can identify different epimorphisms $\varphi,\psi:H\rightarrow
L$ by indicating
\[
Z=\text{Ker\thinspace}(\varphi)\quad\text{ and }\quad L=K/\ker(\psi).
\]

Therefore, to identify $\mathscr H$ we will
write
\begin{equation}
\mathscr H=:H{\prescript{Z}{}\times_{L}^{m}}K~, \label{eq:amalg}%
\end{equation}
where $H$ and $Z$ are subgroups of $I$ and $m$ is a number used to identify groups in different conjugacy classes.
In the case that all the epimorphisms $\varphi$ with the kernel $Z$ are
conjugate, there is no need to use the number $m$ in \eqref{eq:amalg}, so we
will simply write $\mathscr H=H{\prescript{Z}{}\times_{L}K}$. In addition, in
the case that all epimorphisms $\varphi$ from $H$ to $L$ are conjugate, we can also
omit the symbol $Z$, i.e. we will write $\mathscr H=H\times_{L}K$.

\subsection{Bifurcation Theorem}

\begin{theorem}
\label{th:main} Assume that the critical numbers $\lambda_{j,1}\in\Lambda$,
$j=1,2,\dots,46$, for the system \eqref{eq:bif1} are isotypically non-resonant. Then, there
exist multiple global bifurcations of solutions from the critical number $\lambda_{j,1}$ corresponding to the
irreducible representation $V_{n_{j}}$ in Table 3:

\begin{itemize}
\item For $n_{j}=1$ there exists a $G$-orbit of a branch of periodic solutions
with the orbit type ${(\amal{{A_5^p}}{D_{1}}{}{}{})}$;

\item For $n_{j}=2$ there exist $G$-orbits of branches of periodic solutions
with the orbit types ${(\amal{{D_3^p}}{D_{2}}{\bz_{2}}{{\bz_3^p}}{})}$,
${(\amal{{V_4^p}}{D_{2}}{\bz_{2}}{{\bz_2^p}}{})}$,
${(\amal{{A_4^p}}{D_{1}}{}{}{})}$, ${(\amal{{D_3^p}}{D_{1}}{}{}{})}$,
${(\amal{{D_5^p}}{D_{5}}{D_{5}}{{\bz_1^p}}{{1}})}$,
${(\amal{{D_5^p}}{D_{5}}{D_{5}}{{\bz_1^p}}{{2}})}$,
${(\amal{{D_3^p}}{D_{3}}{D_{3}}{{\bz_1^p}}{})}$;

\item For $n_{j}=3$ there exist $G$-orbits of branches of periodic solutions
with the orbit types ${(\amal{{V_4^p}}{D_{2}}{\bz_{2}}{{\bz_2^p}}{})}$,
${(\amal{{D_5^p}}{D_{1}}{}{}{})}$, ${(\amal{{D_3^p}}{D_{1}}{}{}{})}$,
${(\amal{{D_5^p}}{D_{5}}{D_{5}}{{\bz_1^p}}{{1}})}$,
${(\amal{{D_5^p}}{D_{5}}{D_{5}}{{\bz_1^p}}{{2}})}$,
${(\amal{A_4^p}{\bz_{3}}{\bz_3}{V_4^p}{})}$;

\item For $n_{j}=4$ there exist $G$-orbits of branches of periodic solutions
with the orbit types ${(\amal{{D_5^p}}{D_{2}}{\bz_{2}}{{\bz_5^p}}{})}$,
${(\amal{{D_3^p}}{D_{2}}{\bz_{2}}{{\bz_3^p}}{})}$,
${(\amal{{V_4^p}}{D_{2}}{\bz_{2}}{{\bz_2^p}}{})}$,
${(\amal{{D_5^p}}{D_{5}}{D_{5}}{{\bz_1^p}}{{1}})}$,
${(\amal{{D_3^p}}{D_{3}}{D_{3}}{{\bz_1^p}}{})}$,

\item For $n_{j}=5$ there exist $G$-orbits of branches of periodic solutions
with the orbit types ${(\amal{{D_5^p}}{D_{2}}{\bz_{2}}{{\bz_5^p}}{})}$,
${(\amal{{D_3^p}}{D_{2}}{\bz_{2}}{{\bz_3^p}}{})}$,
${(\amal{{V_4^p}}{D_{2}}{\bz_{2}}{{\bz_2^p}}{})}$,
${(\amal{{D_5^p}}{D_{5}}{D_{5}}{{\bz_1^p}}{{2}})}$,
${(\amal{{D_3^p}}{D_{3}}{D_{3}}{{\bz_1^p}}{})}$;

\item For $n_{j}=-1$ here exists a $G$-orbit of a branch of periodic solutions
with the orbit type ${(\amal{{A_5^p}}{D_{2}}{\bz_{2}}{{A_5}}{})}$;

\item For $n_{j}=-2$ there exist $G$-orbits of branches of periodic solutions
with the orbit types ${(\amal{{A_4^p}}{D_{2}}{\bz_{2}}{{A_4}}{})}$,
${(\amal{{D_3^p}}{D_{2}}{\bz_{2}}{{D_3^z}}{})}$,
${(\amal{{D_3^p}}{D_{2}}{\bz_{2}}{{D_3}}{})}$,\break
${(\amal{{V_4^p}}{D_{2}}{\bz_{2}}{{V_4^z}}{})}$,
${(\amal{{D_5^p}}{D_{10}}{D_{10}}{{\bz_1}}{{1}})}$,
${(\amal{{D_5^p}}{D_{10}}{D_{10}}{{\bz_1}}{{2}})}$,
${(\amal{{D_3^p}}{D_{6}}{D_{6}}{{\bz_1}}{})}$,

\item For $n_{j}=-3$ there exist $G$-orbits of branches of periodic solutions
with the orbit types ${(\amal{{D_5^p}}{D_{2}}{\bz_{2}}{{D_5}}{})}$,
${(\amal{{D_3^p}}{D_{2}}{\bz_{2}}{{D_3}}{})}$,
${(\amal{{V_4^p}}{D_{2}}{\bz_{2}}{{V_4^z}}{})}$,\break
${(\amal{{D_5^p}}{D_{10}}{D_{10}}{{\bz_1}}{{1}})}$,
${(\amal{{D_5^p}}{D_{10}}{D_{10}}{{\bz_1}}{{2}})}$,
${(\amal{A_4^p}{\bz_{6}}{\bz_6}{V_4}{})}$;

\item For $n_{j}=-4$ there exist $G$-orbits of branches of periodic solutions
with the orbit types ${(\amal{{D_5^p}}{D_{2}}{\bz_{2}}{{D_5^z}}{})}$,
${(\amal{{D_3^p}}{D_{2}}{\bz_{2}}{{D_3^z}}{})}$,
${(\amal{{V_4^p}}{D_{2}}{\bz_{2}}{{V_4^z}}{})}$,\break
${(\amal{{D_5^p}}{D_{10}}{D_{10}}{{\bz_1}}{{1}})}$,
${(\amal{{D_3^p}}{D_{6}}{D_{6}}{{\bz_1}}{})}$;

\item For $n_{j}=-5$ there exist $G$-orbits of branches of periodic solutions
with the orbit types ${(\amal{{D_5^p}}{D_{2}}{\bz_{2}}{{D_5^z}}{})}$,
${(\amal{{D_3^p}}{D_{2}}{\bz_{2}}{{D_3^z}}{})}$,
${(\amal{{V_4^p}}{D_{2}}{\bz_{2}}{{V_4^z}}{})}$,\break
${(\amal{{D_5^p}}{D_{10}}{D_{10}}{{\bz_1}}{{2}})}$,
${(\amal{{D_3^p}}{D_{6}}{D_{6}}{{\bz_1}}{})}$.
\end{itemize}
\end{theorem}

\begin{proof}
The critical numbers for system \eqref{eq:bif1} are $\lambda_{j,l}=l/\sqrt
{\mu_{j}}$ for $l=1,2,3,\dots,$ and $j=1,2,\dots,46$, where $\mu_{j}$
(together with its isotypical types) are listed in Table \ref{tab:crit-values}. We assumed (under the numerical evidence) that the critical frequencies $\lambda_{j,1}^{-1}$ are
isotypically non-resonant. Based on the ideas explained in section \ref{sec:1}, then
$\lambda_{o}:=\lambda_{j_{o},1}$ is an isolated critical point in $\Lambda$.
That is, there are $\lambda_{-}<\lambda_{o}<\lambda_{+}$ such that
$[\lambda_{-},\lambda_{+}]\cap\Lambda=\{\lambda_{o}\}$. Moreover, there exists
an isolating $G$-neighborhood $\mathscr U$ of $G(u_{o})$ such that no other
critical orbits of $J_{\lambda_{\pm}}$ belong to $\overline{\mathscr U}$.
Thus, we can define the topological invariant $\omega_{G}(\lambda_{o})$ by
\eqref{eq:top-inv}. Then, by the properties of the gradient equivariant degree,
if
\[
\omega_{G}(\lambda_{o})=n_{1}(H_{1})+n_{2}(H_{2})+\dots+n_{m}(H_{m})
\]
is non-zero, i.e. $n_{j}\not =0$ for a $j=1,2,\dots,m$, then there exists a
bifurcating branch of nontrivial solutions to \eqref{eq:bif1} from the orbit
$\{\lambda_{o}\}\times G(u_{o})$ with symmetries at least $(H_{j})$.

Next, by Theorem \ref{thm:SCP},
\[
\nabla_{G}\text{-deg}(\nabla J_{\lambda_{\pm}},\mathscr U)=\Theta\left(
\nabla_{G_{u_{o}}}\text{-deg}(\nabla{ J}_{\lambda_{\pm}},\mathscr
U\cap\mathscr S_{o})\right)  ,
\]
where $G=I\times O(3)\times O(2)$, $G_{u_{o}}=\tilde{I}\times O(2)$ and
$\Theta:U(G_{u_{o}})\rightarrow U(G)$ is the homomorfism given by
$\Theta(H)=(H)$. For convenience, in what follows we will ignore the symbol
$\Theta$. Moreover, by standard linearization technique, we have
\[
\nabla_{G_{u_{o}}}\text{-deg}(\nabla{J}_{\lambda_{\pm}},\mathscr U\cap
\mathscr
S_{o})=\nabla_{G_{u_{o}}}\text{-deg}(\mathscr A_{\lambda_{\pm}},\mathscr
U\cap\mathscr S_{o}).
\]
By \eqref{eq:lin-GdegGrad}, since all the eigenvalues $\mu_{j}$ are
isotopically simple, we have
\begin{align*}
\nabla_{G_{u_{o}}}\text{-deg}(\mathscr A_{\lambda_{-}},\mathscr U\cap
\mathscr S_{o})  &  =\prod_{\left\{  \left(  j,l\right)  \in\mathbb{N}%
^{2}:\lambda_{j,l}<\lambda_{o}\right\}  }\nabla\text{-deg}_{\mathcal{V}%
_{n_{j},l}},\\
\nabla_{G_{u_{o}}}\text{-deg}(\mathscr A_{\lambda_{+}},\mathscr U\cap
\mathscr S_{o})  &  =\nabla\text{-deg}_{\mathcal{V}_{n_{j_{o}},l}}%
\prod_{\left\{  \left(  j,l\right)  \in\mathbb{N}^{2}:\lambda_{j,l}%
<\lambda_{o}\right\}  }\nabla\text{-deg}_{\mathcal{V}_{n_{j},l}},
\end{align*}
where $\nabla\text{-deg}_{\mathcal{V}_{n_{j},l}}$ are gradient $I\times
O(2)$-equivariant basic degrees listed in Appendix \ref{sec:basic}. Therefore,
we obtain
\[
\omega_{G}(\lambda_{o}):=\Big((I\times O(2))-\nabla\text{-deg}_{\mathcal{V}%
_{n_{j_{o}},1}}\Big)\prod_{\left\{  \left(  j,l\right)  \in\mathbb{N}%
^{2}:\lambda_{j,l}<\lambda_{o}\right\}  }\nabla\text{-deg}_{\mathcal{V}%
_{n_{j},l}},
\]
where $(I\times O(2))$ is the unit element in $U(I\times O(2))$. For instance,
the first equivariant invariants are given by
\begin{align*}
\omega_{G}(\lambda_{1,1})  &  =(I\times O(2))-\nabla\text{-deg}_{\mathcal{V}%
_{3,1}}\\
\omega_{G}(\lambda_{2,1})  &  =\nabla\text{-deg}_{\mathcal{V}_{3,1}}%
\ast\Big((I\times O(2))-\nabla\text{-deg}_{\mathcal{V}_{-3,1}}\Big)\\
\omega_{G}(\lambda_{3,1})  &  =\nabla\text{-deg}_{\mathcal{V}_{3,1}}\ast
\nabla\text{-deg}_{\mathcal{V}_{-3,1}}\ast\Big((I\times O(2))-\nabla
\text{-deg}_{\mathcal{V}_{2,1}}\Big).
\end{align*}

We will prove that a maximal orbit type $(H)$ that appears in the gradient
$I\times O(2)$-basic degree $\nabla\text{-deg}_{\mathcal{V}_{n_{j_{o}},1}}$
with non-zero coefficient $n_{H}$,
\[
\nabla\text{-deg}_{\mathcal{V}_{n_{j_{o}},1}}=(I\times O(2))+n_{H}(H)+\dots,
\]
also appears in $\omega_{G}(\lambda_{o})$ with non-zero coefficients.
Hereafter, dots indicate all the remaining terms corresponding to orbit types
strictly smaller than $(H)$. Notice that all such maximal orbit types (which
are indicated in subsection \ref{sec:basic} by red color) belong to $\Phi
_{0}(I\times O(2))$ (i.e. dim\thinspace$W(H)=0$) except for
$(\amal{A_4^p}{\bz_{3}}{\bz_3}{V_4^p}{})$ (in $\nabla\text{-deg}%
_{\mathcal{V}_{3,1}}$) and $(\amal{A_4^p}{\bz_{6}}{\bz_6}{V_4}{}))$ (in
$\nabla\text{-deg}_{\mathcal{V}_{-3,1}}$).

Now, assume that $(H)$ is a maximal orbit type such that dim\thinspace$W(H)=0$
and
\[
\nabla\text{-deg}_{\mathcal{V}_{n,1}}=(I\times O(2))+n_{H}(H)+\dots,
\]
with $n_{H}\not =0$. By maximality of $(H)$ in $\mathcal{V}_{n,1}$, formula
\eqref{eq:rec-Brouwer} gives
\[
n_{H}=\frac{(-1)^{k}-1}{|W(H)|},\quad k:=\text{dim\thinspace}\mathcal{V}%
_{n,1}^{H}.
\]
Then $k$ must be odd and consequently $n_{H}=-1$ when $|W(H)|=2$ or $n_{H}=-2$
when $|W(H)|=1$. Suppose now that $\nabla\text{-deg}_{\mathcal{V}_{\bar{n},1}%
}$ is another (not necessarily different) basic degree containing a non-zero
coefficient for $(H)$, i.e.
\[
\nabla\text{-deg}_{\mathcal{V}_{\bar{n},1}}=(I\times O(2))+n_{H}(H)+\dots.
\]
Then, we have,
\[
\nabla\text{-deg}_{\mathcal{V}_{n,1}}\ast\nabla\text{-deg}_{\mathcal{V}%
_{\bar{n},1}}=(I\times O(2))+2n_{H}(H)+n_{H}^{2}(H)^{2}+\dots
\]
However, by \eqref{eq:rec-coef}, we have $(H)\ast(H)=m_{H}(H)+\dots$, where
\[
m_{H}:=\frac{|W(H)|\cdot|W(H)|}{|W(H)|}=|W(H)|.
\]
Thus
\[
\nabla\text{-deg}_{\mathcal{V}_{n,1}}\ast\nabla\text{-deg}_{\mathcal{V}%
_{\bar{n},1}}=(I\times O(2))+\left(  2n_{H}+n_{H}^{2}m_{H}\right)  (H)+\dots.
\]
One can easily notice that
\[
2n_{H}+n_{H}^{2}m_{H}=0,
\]
for both cases $n_{H}=-1$ and $n_{H}=-2$. Therefore, the coefficient of the
product $\nabla\text{-deg}_{\mathcal{V}_{n,1}}\ast\nabla\text{-deg}%
_{\mathcal{V}_{\bar{n},1}}$ is zero for the group $(H)$. Consequently, since
ether $\nabla_{G_{u_{o}}}\text{-deg}(\mathscr A_{\lambda_{-}},\mathscr
U\cap\mathscr S_{o})$ or $\nabla_{G_{u_{o}}}\text{-deg}(\mathscr
A_{\lambda_{+}},\mathscr U\cap\mathscr S_{o})$ (but not both) contains an even
number of factors $\nabla\text{-deg}_{\mathcal{V}_{n_{i},1}}$ with non-zero
coefficient $n_{H}$ of $(H)$, it follows that their difference also contains
non-zero $\pm n_{H}$ coefficient of $(H)$. Actually, the computation with GAP shows
that $n_{H}=-1$, then in these cases we have%
\[
\omega_{G}(\lambda_{o})=\pm(H)+\dots.\text{.}%
\]

Now, assume that $H=\amal{A_4^p}{\bz_{3}}{\bz_3}{V_4^p}{}$ and consider
$\nabla\text{-deg}_{\mathcal{V}_{3,1}}=(G)+n_{H}(H)+\dots$. Then we have
\[
\nabla\text{-deg}_{\mathcal{V}_{3,1}}\ast\nabla\text{-deg}_{\mathcal{V}_{3,1}%
}=(G)+2n_{H}(H)+n_{H}^{2}(H)^{2}+\dots.
\]
By functoriality property of the gradient equivariant degree we have that the
inclusion $\psi:I\times S^{1}\rightarrow I\times O(2)$ induces the Euler
homomorphism $\Psi:U(I\times O(2))\rightarrow U(I\times S^{1})$ such that
$\Psi(\nabla\text{-deg}_{\mathcal{V}_{3,1}})$ is also a gradient equivariant
basic degree (see \cite{DaKr}). This can be easily computed (cf. \cite{RR}) as
follows,
\begin{align*}
\Psi(\nabla\text{-deg}_{\mathcal{V}_{3,1}}) &  =(I\times S^{1})-(D_{5}%
^{p})-(D_{3}^{p})-(A_{4}^{t_{1}}\times\bz_{2})-(A_{4}^{t_{2}}\times\bz_{2})\\
&  -(V_{4}^{-}\times\bz_{2})-(\bz_{5}^{t_{1}}\times\bz_{2})-(\bz_{5}^{t_{2}%
}\times\bz_{2})+2(\bz_{2}^{p}).
\end{align*}
Thus $n_{H}=-1$. Notice that by \eqref{eq:Euler-hom}, $\Psi
(\amal{A_4^p}{\bz_{3}}{\bz_3}{V_4^p}{})=(A_{4}^{t_{1}}\times\bz_{2}%
)+(A_{4}^{t_{2}}\times\bz_{2})$. As it was shown in \cite{RR}, $(A_{4}^{t_{i}%
}\times\bz_{2})\ast(A_{4}^{t_{j}}\times\bz_{2})=0$. Thus we have
\[
0=\Psi((H)\ast(H))=\Psi(m_{H}(H)+\dots)=m_{H}\Big(((A_{4}^{t_{1}}\times
\bz_{2})+(A_{4}^{t_{2}}\times\bz_{2})\Big),
\]
which implies $m_{H}=0$. Therefore, $(H)\ast(H)=0$ and for $k\in\mathbb{N}$,
\[
\left(  \nabla\text{-deg}_{\mathcal{V}_{3,1}}\right)  ^{k}=(I\times
O(2))-k(H)+\dots.
\]
Clearly,
$\omega_{G}(\lambda_{o})$ has a non-zero coefficient,
\[
\omega_{G}(\lambda_{o})=(H)+\dots.
\]
For $(H)=(\amal{A_4^p}{\bz_{6}}{\bz_6}{V_4}{})$ the proof is similar. This
concludes the proof of our main theorem.
\end{proof}

\begin{remark}\rm
Since all the invariants are $\omega_{G}(\lambda_{o})=(H)+...$ for
$H=(\amal{A_4^p}{\bz_{3}}{\bz_3}{V_4^p}{})$ and
$H=(\amal{A_4^p}{\bz_{6}}{\bz_6}{V_4}{})$, then the sum of $\omega_{G}$'s can
never be zero, i.e. all the connected components $\mathcal{C}$ with symmetries
$(\amal{A_4^p}{\bz_{3}}{\bz_3}{V_4^p}{})$ and
$(\amal{A_4^p}{\bz_{6}}{\bz_6}{V_4}{})$ are non-compact. Similarly, notice that
there is an odd number of irreducible subrepresentations $\mathcal V_{-n}$ in the isotypical component $\mathscr V_{-n}$, for $n=1$, $3$, $4$, $5$,
and the topological invariant is $\omega_{G}(\lambda_{o})=\pm(H)+...$ (for
a maximal group $(H)$). This excludes a possibility that all the branches with the orbit type $(H)$, bifurcating from all the critical points $\lambda_{j,1}$ corresponding to $\mathcal V_{-n}$, are compact.  Thus, for any maximal orbit type $(H)$ in   $\mathcal V_{-n}$ for $n=1,3,4,5$  there exists a non-compact branch   $\mathcal{C}$ with orbit type $(H)$.
\end{remark}

\begin{remark}\rm
All the gradient basic degrees $\nabla\text{-deg}_{\mathcal{V}_{\pm
n,1}}$, which were computed using G.A.P. programming, are included in Appendix
\ref{sec:equi-degree}. These degrees can be used to compute the exact value of topological
invariants $\omega_{G}(\lambda_{o})$ even in the case that $\lambda_{o}$ is isotypically 
resonant, so a bifurcation  result can be established in the resonant case as well. For example, such a resonant case was studied in \cite{BeGa} (to classify the nonlinear modes in a tetrahedral molecule). 
\end{remark}

\section{Description of Symmetries and Numerical results}

\label{sec:symmetries}

For any maximal orbit type $(H)$ in $\mathcal V_{n,1}$ the element $-1\in\mathbb{Z}_{2}<I$ belongs to $ H$, and in $\mathcal  V_{-n,1}$, the element $(-1,-1)\in\mathbb{Z}_{2}\times  \bz_2<I\times S^1$ belongs to $ H$. A solution  $u(t)$ in the fixed point space for a group containing $-1\in\mathbb{Z}_{2}$, satisfies
\[
u_{\tau,k}(t)=-u_{\tau^{-1},k}(t),
\]
while for a group containing $(-1,-1)\in\mathbb{Z}_{2}\times  \bz_2$ satisfies
\[
u_{\tau,k}(t)=-u_{\tau^{-1},k}(t+\pi).
\]
Since the isotropy groups in $\mathcal V_{n,1}$ and $\mathcal V_{-n,1}$ differ only in this element, we only need to describe the symmetries of the maximal groups for
the representations $\mathcal{V}_{n,1}$. 

The existence of the symmetry $\kappa\in O(2)$ in the maximal groups implies
that the solutions are brake orbits,%
\[
u_{\tau,k}(t)=u_{\tau,k}(-t)\text{,}%
\]
i.e., the velocity $\dot{u}$ of all the molecules are zero at times
$t=0,\pi$, i.e., $\dot{u}(0)=\dot{u}(\pi)=0$. We classify the maximal groups in two
classes: the groups that have the element $\kappa\in O(2)$ and the groups that
have the element $\kappa$ coupled with a rotation of $I$. That is, if there is an
element $\gamma\in\mathcal{C}_{2}$ such that $(\gamma,\kappa)$ is in the
second class of groups, then their solutions have the symmetry
\[
u_{\tau,k}(t)=\rho(\gamma)u_{\gamma\tau\gamma^{-1},\gamma^{-1}(k)}(-t).
\]

The maximal orbit type  that does not have a symmetry $(\gamma,\kappa)$ is
$(\amal{A_4^p}{\bz_{6}}{\bz_6}{V_4}{})$ which is the only maximal group (in $\mathscr V_{3,1}$)
with Weyl group of dimension one.
\vs

\subsection{Standing Waves (Brake Orbits)}

In this category we consider the groups that have the element $\kappa\in
O(2)$, which generate the group $D_{1}<O(2)$. 

For the groups%
\[
{({A_{5}^{p}}\prescript{}{}\times_{{}}^{{}}D_{1}),({A_{4}^{p}}%
\prescript{}{}\times_{{}}^{{}}D_{1}),({D_{5}^{p}}\prescript{}{}\times_{{}}%
^{{}}D_{1}),({D_{3}^{p}}\prescript{}{}\times_{{}}^{{}}D_{1})}%
\]
the solutions have the following symmetries at all times: icosahedral symmetries for 
${{A_{5}}}$, tetrahedral symmetries for ${{A_{4}}}$, pentagonal
symmetries for ${{D_{5}}}$ and triangular symmetries for
${{D_{3}}}$.

For the group%
\[
{({D_{3}^{p}}\prescript{{\bz_3^p}}{}\times_{\mathbb{Z}_{2}}^{{}}D_{2})}%
\]
the solutions are symmetric by the $2\pi/3$-rotations of $\mathbb{Z}_{3}%
<D_{3}<I$, while the reflection of $D_{3}<I$ is coupled with the $\pi$-time
shift of $-1\in\mathbb{Z}_{2}<S^{1}$. Therefore, the solutions in three faces have the exact dynamics, but these faces are not symmetric by reflection
such as in the symmetries of ${({D_{3}^{p}}\prescript{}{}\times_{{}%
}^{{}}D_{1})}$.

For the group%
\[
{({V_{4}^{p}}\prescript{{\bz_2^p}}{}\times_{\mathbb{Z}_{2}}^{{}}D_{2}%
)}\text{,}%
\]
the solutions are symmetric by the $\pi$-rotations of $V_{4}<I$, while the
other $\pi$-rotation of $\mathbb{Z}_{2}<V_{4}$ is coupled with the $\pi$-time
shift of $-1\in D_{2}<S^{1}$.

These seven symmetries give solutions which are standing waves in the sense
that each symmetric face has the exact dynamic repeated for all times.

\subsection{Discrete Rotating Waves}

In the groups%
\[
{({D_{5}^{p}}\prescript{{\bz_1^p}}{}\times_{D_{5}}^{1}D_{5})},{({D_{5}^{p}%
}\prescript{{\bz_1^p}}{}\times_{D_{5}}^{2}D_{5})}~,
\]
the spatial dihedral group $D_{5}<I$ is coupled with the temporal group
$D_{5}<O(2)$. Therefore, in these solutions we have $5$ faces with the same
dynamics, but there is a $2\pi/5$-time shift in time between consecutive
faces. In addition, $\kappa$ is coupled with a $\pi$-rotation, i.e., there is
an axis of symmetry in each face. In this sense, the solutions have the
appearance of a discrete rotating wave with a $2\pi/5$ delay along consecutive faces. There are two groups
because there are two different conjugacy classes, $\mathcal{C}_{4}$ and
$\mathcal{C}_{5}$, of $A_{5}$.

Similarly, in the solutions of the group%
\[
{({D_{3}^{p}}\prescript{{\bz_1^p}}{}\times_{D_{3}}^{{}}D_{3})~,}%
\]
we have $3$ faces with the same dynamics, but with a $2\pi/3$-time shift, i.e.,
the solutions have the appearance of a discrete rotating wave in $3$ faces
with a $2\pi/3$-time shift and each face has an axis of symmetry.

For the solutions of the group
\[
(\amal{A_4^p}{\bz_{6}}{\bz_6}{V_4}{})
\]
we have $3$ faces with the same dynamics with a $2\pi/3$-time shift.
Moreover, in these solutions the inversion is coupled with a  $\pi$-time shift
in time. Therefore, there are a total of $6$ faces ($3$ faces and their
inversions) that have the same dynamics but with  $2\pi/6$-time shift. In these
solutions the faces do not have an axis of symmetry, instead there are two
symmetries by a $\pi$-rotation.

\subsection{Numerical results}

In this section, we present the implementation of the numerical continuation of some families of periodic solutions. In order to compute numerically the families of periodic solutions, we use the
Hamiltonian formulation,%
\begin{equation}
\dot{x}=J\nabla H(x),\qquad x=(q,p),\label{ODE}%
\end{equation}
where $H(q,p)=\left\vert p\right\vert ^{2}/2-V(q)$ is the Hamiltonian and $J$
is the symplectic matrix%
\[
J=\left(
\begin{array}
[c]{cc}%
0 & -I\\
I & 0
\end{array}
\right)  \text{.}%
\]

Since the Hamiltonian is invariant by the action of the group $\mathbb{R}%
^{3}$ that acts by translation, $O(3)$ by rotations and $\varphi\in S^{1}$ by
time shift, then the Hamiltonian satisfies the orthogonal relations%
\[
\left\langle H(x),A_{j}(x)\right\rangle =0,
\]
for $j=1,...,7$, where $A_{j}$ are the generators of the groups,%
\begin{align*}
A_{j}(q,p) &  =\partial_{\tau}|_{\tau=0}(q+\tau\mathcal{E}_{j},p)=(\mathcal{E}%
_{j},0),\qquad\mathcal{E}_{j}=(e_{j},...,e_{j})\\
A_{j+3}(q,p) &  =\partial_{\theta}|_{\theta=0}(e^{\theta\mathcal{J}%
}q,e^{\theta\mathcal{J}}p)=(\mathcal{J}_{j}q,\mathcal{J}_{j}p),\qquad
\mathcal{J}_{j}=diag(J_{j},...,J_{j})
\end{align*}
for $j=1,2,3$, and%
\[
A_{7}(q,p)=\partial_{\varphi}|_{\varphi=0}(q,p)(t+\varphi)=J\nabla H\text{.}%
\]

\begin{remark}
Actually, the conserved quantities $G_{j}$ are related to the generator fields
$A_{j}$ by
\[
A_{j}=J\nabla G_{j}.
\]
Using the Poisson bracket, the orthogonality relations are equivalent to%
\[
\{H,G_{j}\}=\left\langle \nabla H,J\nabla G_{j}\right\rangle =\left\langle
\nabla H,A_{j}\right\rangle =0\text{.}%
\]
The explicit conserved quantities are $G_{j}=-p\cdot\mathcal{E}_{j}$,
$G_{j+3}=p^{T}\mathcal{J}_{j}q$, for $j=1,2,3$, and $G_{7}=H$.
\end{remark}

To numerically continue a solution it is necessary to augment the
differential equation with Lagrange multipliers $\lambda_{j}\in\mathbb{R}$ for
$j=1,..,7$,%
\begin{equation}
\dot{x}=J\nabla H(x)+\sum_{j=1}^{7}\lambda_{j}JA_{j}(x)\text{.}\label{ODE2}%
\end{equation}
The solutions of equation \eqref{ODE2} are solutions of the original equations
of motion when the values of the seven parameters are zero. If $A_{j}(x)$ are
linearly independent, a solution $x$ to the equation (\ref{ODE2}) is a
solution to the equation (\ref{ODE}) because 
\[
0=\left\langle \dot{x},JA_{i}(x)\right\rangle =\sum_{j=1}^{7}\lambda
_{j}\left\langle A_{j},A_{i}\right\rangle 
\]
implies that $\lambda_{j}=0$ for $j=1,...,7$.

The period $T=2\pi\lambda$ can be obtained as parameter in equation
\eqref{ODE2} by rescaling time,%
\[
\dot{x}=TJ\nabla H+\sum_{j=1}^{7}\lambda_{j}JA_{j}\text{.}%
\]
Let $\varphi_{t}(x)$ be the flow of this equation. We can define the time one
map (for the rescaled time)\ as
\[
\varphi_{1}(x;\lambda_{1},...,\lambda_{7},T):V\times\mathbb{R}^{7}%
\times\mathbb{R}\rightarrow V\text{,}%
\]
where the period $T$ is a parameter. Therefore, a fixed point of $\varphi
_{1}(x)$ corresponds to a $T$-periodic solutions of the Hamiltonian system.

To numerically continue the fixed points of $\varphi_{1}(x)$ it is necessary
to implement Poincar\'{e} sections. For this we define the augmented map%
\begin{align*}
F(q,p,\lambda_{1},...,\lambda_{7};T) &  :V\times\mathbb{R}^{7}\times
\mathbb{R}\rightarrow V\times\mathbb{R}^{7}\\
F &  =\left(  x-\varphi_{1}(x),A_{j}(\tilde{x})\cdot\left(  x-\tilde
{x}\right)  \right)  .
\end{align*}
Then a solution of $F=0$ is a $T$-periodic solution of the Hamiltonian system.
The restrictions $A_{j}(\tilde{x}%
)\cdot\left(  x-\tilde{x}\right)  =0$ for $j=1,...,7$ represent the
Poincar\'{e} sections, where $\tilde{x}$ is a previously computed solutions on
the family of solutions of $F=0$. This map is a local submersion except for bifurcation points, see \cite{MuAl}.


\begin{figure}[H]
\begin{center}
\resizebox{11.0cm}{!}{\includegraphics{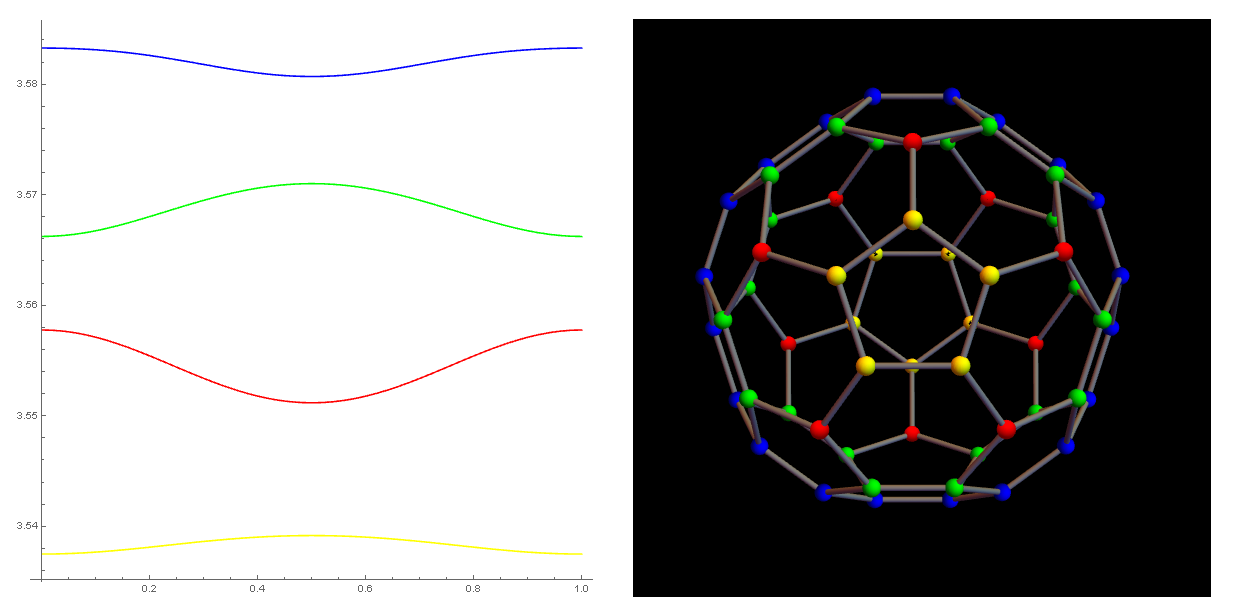} } 
\resizebox{11.0cm}{!}{\includegraphics{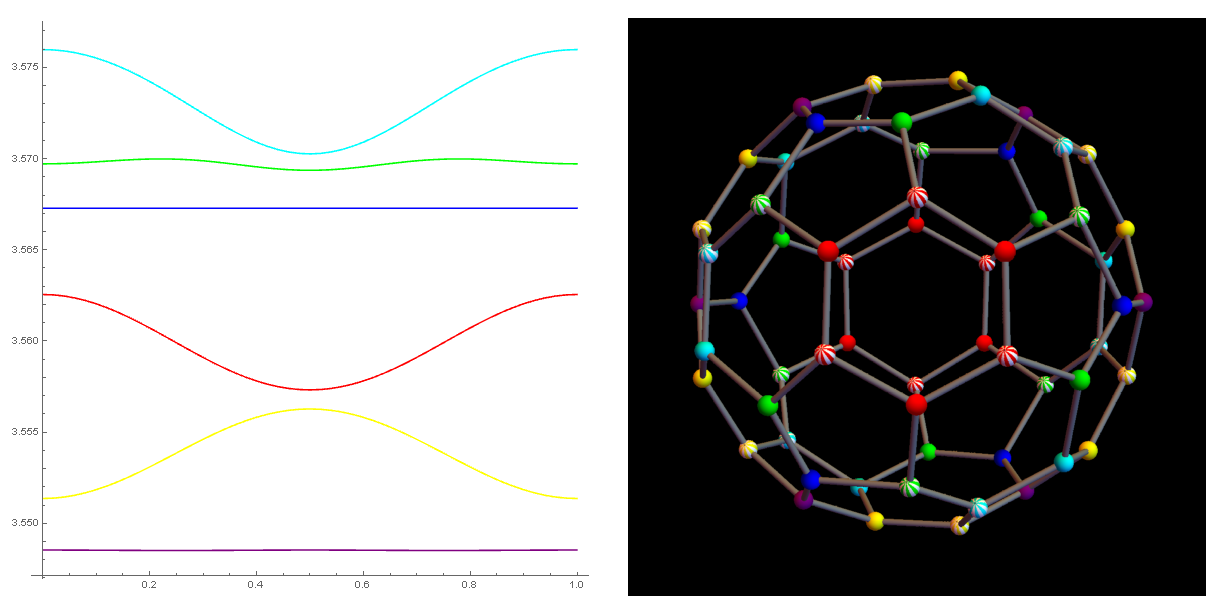} }
 \resizebox{11.0cm}{!}{\includegraphics{e03m1.png} }
\end{center} 
\caption{  Top: Solutions from eigenvalue  $j=2$  with symmetries $({D_{5}^{p}}\prescript{}{}\times_{{}}
^{{}}D_{1})$.
Middle: Solution from eigenvalue  $j=3$  with symmetries $
{({D_{3}^{p}}\prescript{{\bz_3^p}}{}\times_{\mathbb{Z}_{2}}^{{}}D_{2})}$. Bottom: Solution from eigenvalue  $j=3$  with symmetries 
$(\amal{{A_4^p}}{D_{2}}{\bz_{2}}{{A_4}}{})$ }%
\label{fig-1}%
\end{figure}
\begin{figure}[H]
\begin{center}
\resizebox{11cm}{!}{\includegraphics{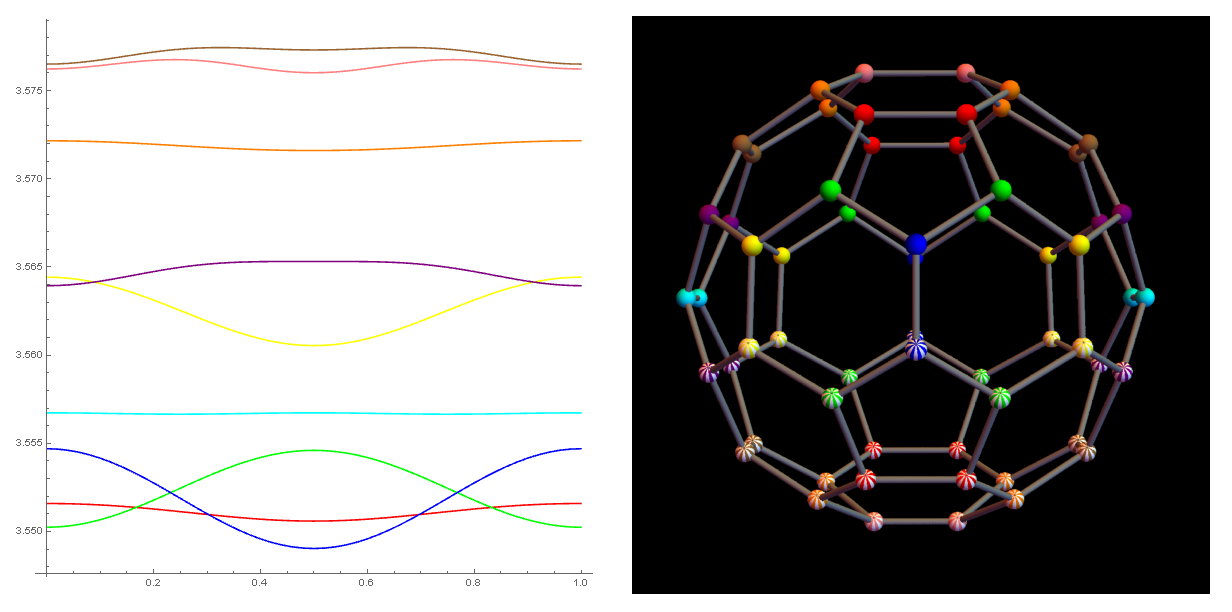} } 
\resizebox{11cm}{!}{\includegraphics{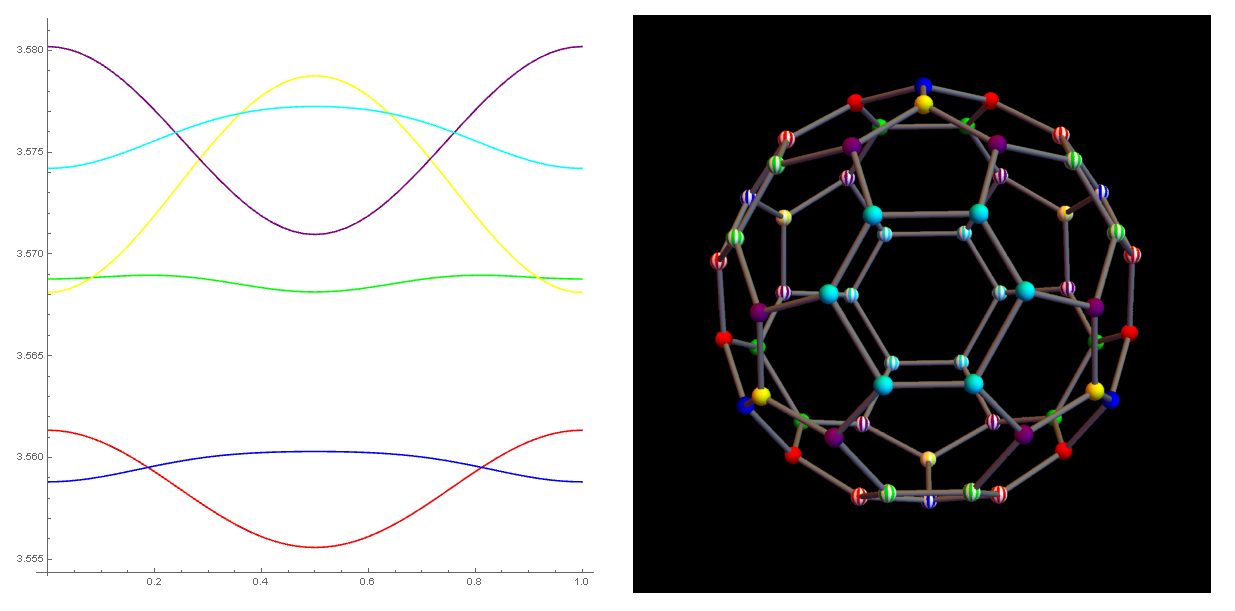} }
 \resizebox{11cm}{!}{\includegraphics{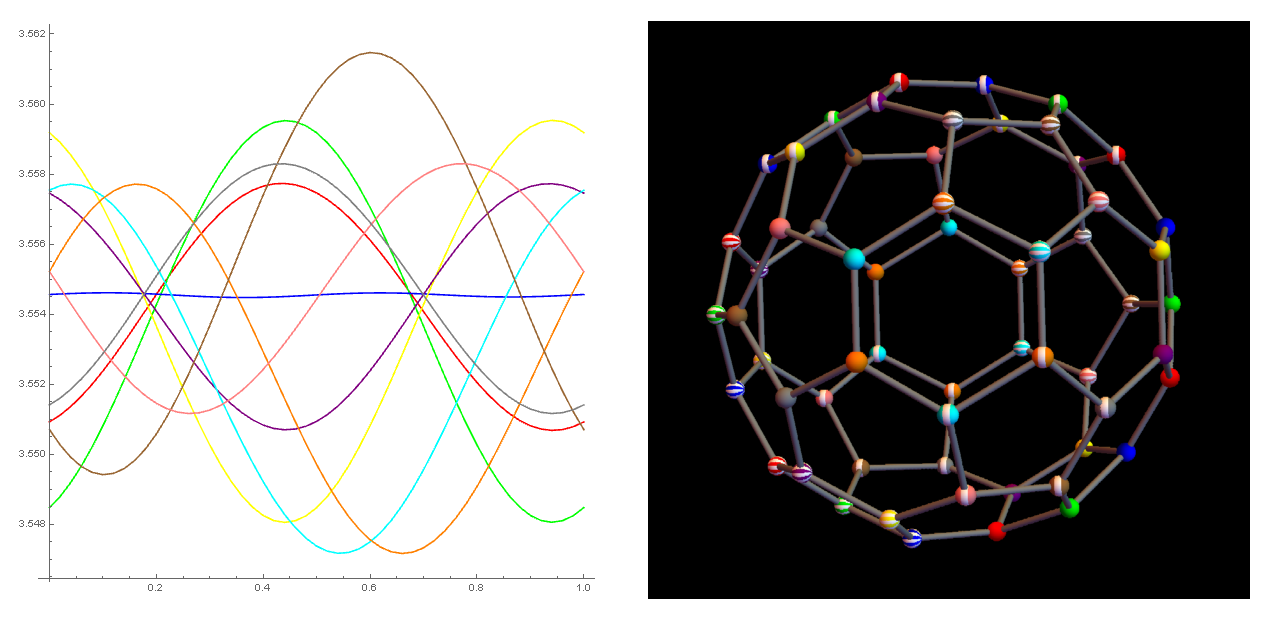} }
\end{center}
\caption{ Top: Solutions from eigenvalue  $j=4$  with symmetries  $(\amal{{V_4^p}}{D_{2}}{\bz_{2}}{{V_4^z}}{})$. Middle: Solutions from eigenvalue  $j=5$  with symmetries $(\amal{{D_3^p}}{D_{6}}{D_{6}}{{\bz_1}}{})$.
 Bottom: Solutions from eigenvalue  $j=9$  with symmetries ${({D_{3}^{p}}\prescript{{\bz_1^p}}{}\times_{D_{3}}^{{}}D_{3})}$. }%
\label{fig-2}%
\end{figure}

The map $\varphi_{1}(x)$ is computed numerically using a Runge-Kutta
integrator.\  A first solution of $F=0$ is obtained by applying a Newton method
in the approximating solution obtained by solving the linearized Hamiltonian system. The family of periodic
solutions is computed numerically using a pseudo-arclength procedure to
continue the solutions of $F=0$.

We present the results of our numerical computations in Figures \ref{fig-1} and \ref{fig-2}.  The position of the atoms in space are in the right columns. The atoms with the same color have oscillations related by a rotation or inversion in $O(3)$. In addition, atoms with the same color but different texture describe oscillations that are related by the inversion coupled with a phase shift in time. In the left columns of Figures \ref{fig-1} and \ref{fig-2} we illustrate the norm of the atoms oscillating in time.

\vskip.3cm
\appendix{\huge \bf Appendix}

\section{Equivariant Gradient Degree}\label{sec:equi-degree}
For an Euclidean space $V$ we denote by $B(V)$ the open unit ball in $V$, and for $x$, $y\in V:=\br^n$ we will denote by $x\bullet y $ the standard inner product in $V$.

We assume that $G$ stands for a compact Lie group and that all considered subgroups of
$G$ are  closed. For a subgroup $H\le G$,
 $N\left(  H\right)  $ stands for  the {\it normalizer} of $H$ in $G$, and  $W\left(
H\right)  =N\left(  H\right)  /H$ denotes  the {\it Weyl group} of $H$ in $G$. We denote by
$\left(  H\right)  $ the {\it conjugacy} class of $H$ in $G$ and use the notations $
\Phi\left(  G\right)     :=\left\{  \left(  H\right)  :H\;\;\text{is a
subgroup of }\;G\right\}$ and 
$\Phi_{n}\left(  G\right)     :=\left\{  \left(  H\right)  \in\Phi\left(
G\right)  :\text{\textrm{dim\thinspace}}W\left(  H\right)  =n\right\}$.
The set $\Phi\left(  G\right)  $ has a natural partial order given by: 
$\left(  H\right)  \leq\left(  K\right)  \;\;\Longleftrightarrow\;\;\exists
_{g\in G}\;\;gHg^{-1}\le K$.

For a $G$-space $X$ and $x\in X$, we put
$
G_{x}:=\left\{  g\in G:\;gx=x\right\}$ to denote
the {\it isotropy group} of $x$, $G\left(  x\right)  :=\left\{  gx:\;g\in G\right\}$ to denote 
the orbit of $x$, and the conjugacy class $\left(  G_{x}\right)  :=\left\{  H\subset G:\;\exists_{g\in G}\;\;G_{x}%
=g^{-1}Hg\right\}$ will be called the orbit type 
 of $x$, and $\Phi(G;X):=\{(G_x): x\in X\}$ will stand for the set of  all the orbit types in $X$. We also put $\Phi_n(G;X):=\Phi(G;X)\cap \Phi_n(G)$.
 
For a subgroup $H\le G$, we write
$X^{H}:=\left\{  x\in X:\;G_{x}\ge H\right\}$
to denote the {\it $H$-fixed point space} of $H$. The orbit space for a $G$-space $X$ will be
denoted by $X/G$.

As  any compact Lie group admits only countably many
non-equivalent real irreducible representations, given a
compact Lie group $G $, we will assume that we have a complete list of all real
 irreducible representations, denoted $\mathcal{V}_{i}$, $i=0 $, $1$, $\ldots $, which could also be identified by the {\it character list} $\{\chi_i\}$. We refer to \cite{BaKr06} for examples of such lists and the related notation. 
\vs

Any finite-dimensional real  $G$-representation $V$ can be
decomposed
into the direct sum of $G $-invariant subspaces
\begin{equation}
V=V_{0}\oplus V_{1}\oplus \dots \oplus V_{r}\text{,}  \label{eq:Giso}
\end{equation}%
called the $G$\textit{-}{\it isotypical decomposition of }$V$, where each isotypical component $V_{i}$  is \emph{modeled} on the irreducible $G$-representation $\mathcal{V}_{i}$, $i=0,$ $1,$ $\dots ,$ $r$,  i.e., $V_{i}$ contains all the irreducible subrepresentations of $V$ which are equivalent to $\mathcal{V}_{i}$.
\vs
Let $V$ be a $G$-representation, $\Omega\subset V$ an open $G$-invariant bounded set and $f:V\to V$ a continuous $G$-equivariant map such that for all $x\in \partial \Omega$ we have $f(x)\not=0$. Then we say that $f$ is an {\it $\Omega$-admissible} $G$-map and we call $(f,\Omega)$ an {\it admissible $G$-pair}. The set of all admissible $G$-pairs in $V$  will be denoted by $\mathcal M^G(V)$. We also put $\mathcal M^G:=\bigcup_{V}\mathcal M^G(V)$ (here $V$ denotes all possible $G$-representations) to denote the set of all admissible $G$-pair. A map $f:V\to V$ is called a {\it gradient map} if there exists a continuously differentiable $\vp:V\to \br$ such that $f=\nabla \vp$. We denote by $\mathcal M_\nabla^G(V)$ the subset of $\mathcal M^G(V)$ consisting of all gradient maps and we define $\mathcal M_\nabla^G:=\bigcup_{V}\mathcal M_\nabla^G(V)$. In the set $\mathcal M^G(V)$ (resp. $\mathcal M_\nabla^G(V)$) we have the so-called  {\it admissible homotopy} (resp. {\it admissible gradient homotopy}) relation between $(f_0,\Omega)$ and $(f_1,\Omega)$, i.e., if $f_1$ and $f_2$ are homotopic by an homotopy $h:[0,1]\times V\to V$, such that $h_t$ belongs to $\mathcal M^G(V)$ (resp. $\mathcal M_\nabla^G(V)$) for every $t\in [0,1]$.

\vs
\subsection{Euler and Burnside Rings}

The concept of the {\it Euler ring} was introduced by T. tom Dieck in \cite{tD}. Due to its topological nature, computations of  the Euler ring $U(G)$, for a general compact group $G$,  may be quite complicated.   However, in our case of interest, when $G:=\Gamma\times O(2)$ with  $\Gamma$ being  a finite group, the Euler ring structure in $U(G)$ can be effectively described  by using elementary  techniques based on the reduction techniques and the properties of the Euler ring homomorphisms (see \cite{DaKr} for more details). 

Let us recall the definition of the Euler ring $U(G)$. As a $\bz$-module, $U(G)$ is the free $\bz$-module generated by $\Phi(G)$, i.e. $U\left(  G\right)  :={\mathbb{Z}}\left[  \Phi\left(  G\right)  \right] $. 
The 
ring multiplication is defined on $U(G)$ on generators $\left(  H\right) $, $\left(  K\right)
\in\Phi\left(  G\right)  $ by
\begin{equation}
\left(  H\right)  \ast\left(  K\right)  =\sum_{\left(  L\right)  \in
\Phi\left(  G\right)  }n_{L}\left(  L\right)  ,\label{eq:Euler-mult}%
\end{equation}
where
\begin{equation}
n_{L}:=\chi_{c}\left(  \left(  G/H\times G/K\right)  _{L}/N\left(  L\right)
\right) \label{eq:Euler-coeff}%
\end{equation}
with $\chi_{c}$ the Euler characteristic taken in Alexander-Spanier cohomology
with compact support (cf. \cite{Spa}). We refer to  \cite{BtD} for more details.

The ${\mathbb{Z}}$-module $A\left(  G\right)  :={\mathbb{Z}}\left[  \Phi
_{0}\left(  G\right)  \right]  $ equipped with the same  multiplication as in
$U\left(  G\right)$ but restricted to generators from $\Phi_{0}\left(
G\right)  $ is called \emph{Burnside ring}, i.e.,
\[
\left(  H\right)  \cdot\left(  K\right)  =\sum_{\left(  L\right)  }%
n_{L}\left(  L\right)  ,\qquad\left(  H\right) ,\,\left(  K\right)
,\,\left(  L\right)  \in\Phi_{0}\left(  G\right)  ,
\]
where $n_L$  stands for the number of $(L)$ orbits in $G/H\times G/K$, i.e.  $n_{L}:=\left(  \left(  G/H\times G/K\right)  _{L}/N\left(  L\right)
\right)  =\left\vert \left(  G/H\times G/K\right)  _{L}/N\left(L\right)
\right\vert $ (here $|X|$ stands for the number of elements in the set $X$). We have the following recurrence formula 
{\small
\begin{equation}
n_{L}=\frac{n\left(  L,K\right)  \left\vert W\left(  K\right)  \right\vert
n\left(  L,\text{ }H\right)  \left\vert W\left(  H\right)  \right\vert
-{\displaystyle \sum_{\left(  \wt L\right)  >\left(  L\right)  }}n\left(  L,\wt L\right)
n_{\wt L}\left\vert W\left(  \wt L\right)  \right\vert }{\left\vert W\left(  L\right)
\right\vert }, \label{eq:rec-coef}%
\end{equation}}
where
\[
n(L,K)=\left\vert \frac{N(L,K)}{N(K)}\right\vert ,\quad N(L,K):=\{g\in
G:gLg^{-1}\subset K\},
\]
and $\left(  H\right)$, $\left(  K\right)$, $(  L)$,  $(
\wt L)  $ are taken from $\Phi_{0}\left(  G\right) $.

Clearly, the structure of the Burnside ring $A(G)$ is significantly simpler and can be effectively computed. It is also possible to implement the G.A.P. routines in computer programs evaluating Burnside rings products. Notice that $A\left(  G\right)  $ is a ${\mathbb{Z}}$-submodule of $U\left(
G\right)  $, but not a subring. However   (see \cite{BKR}), the projection  $\pi_{0}:U\left(  G\right)  \rightarrow
A\left(  G\right)  $ defined on generators $\left(  H\right)  \in\Phi\left(  G\right)
$ by
\begin{equation}\label{eq:pi0}
\pi_{0}\left(  \left(  H\right)  \right)  =%
\begin{cases}
\left(  H\right)  & \text{ if }\;\left(  H\right)  \in\Phi_{0}\left(
G\right)  ,\\
0 & \text{ otherwise,}%
\end{cases}
\end{equation}
is a ring homomorphism, i.e.,
\[
\pi_{0}\left(  \left(  H\right)  \ast\left(  K\right)  \right)  =\pi
_{0}\left(  \left(  H\right)  \right)  \cdot\pi_{0}\left(  \left(  K\right)
\right)  ,\qquad\left(  H\right)  ,\,\left(  K\right)  \in\Phi\left(
G\right)  ,
\]
where `$\cdot$' denotes the multiplication in the Burnside ring $A(G)$.
The homomorphism $\pi_0$ allows to identify the  Burnside ring
$A\left(  G\right)  $ as a part of  the
Euler ring  $U\left(  G\right)$ and with the help of additional algorithms,  the Euler ring structure for $G=\Gamma\times O(2)$ can be completely computed by elementary means (cf.  \cite{DaKr}).

\subsection{Equivariant Gradient Degree}

The existence and properties of the so-called $G$-equivariant gradient degree are presented in the following result from \cite{Geba}:
\begin{theorem}
\label{thm:Ggrad-properties} There exists a unique map $\nabla_{G}\text{\rm -deg\,}:\mathcal{M}_{\nabla}^{G}\rightarrow U(G)$,
which assigns to every $(\nabla\varphi,\Omega)\in\mathcal{M}_{\nabla}^{G}$ an
element $\nabla_{G}\text{\rm -deg\,}(\nabla\varphi,\Omega)\in
U(G)$, called the $G$\textit{-gradient degree} of $\nabla\varphi$ on $\Omega
$,
\begin{equation}
\nabla_{G}\text{\rm -deg\,}(\nabla\varphi,\Omega)=\sum
_{(H)\in\Phi(G)}n_{H}(H_{i})=n_{H_{1}}(H_{1})+\dots+n_{H_{m}%
}(H_{m}), \label{eq:grad-deg}%
\end{equation}
satisfying the following properties:

\begin{description}
\item \textbf{(Existence)} If $\nabla_{G}\text{\rm -deg\,}
(\nabla\varphi,\Omega)\not =0$, i.e. there is in \eqref{eq:grad-deg} a
non-zero coefficient $n_{H_{i}}$, then exists $u_{0}\in\Omega$ such that
$\nabla\varphi(u_{0})=0$ and $(G_{u_{0}})\geq(H_{i})$.

\item \textbf{(Additivity)} Let $\Omega_{1}$ and $\Omega_{2}$ be two disjoint
open $G$-invariant subsets of $\Omega$ such that $(\nabla\varphi)^{-1}%
(0)\cap\Omega\subset\Omega_{1}\cup\Omega_{2}.$ Then, $\nabla_{G}\text{\rm -deg\,}(\nabla\varphi,\Omega)=\nabla_{G}\text{\rm -deg\,}(\nabla\varphi,\Omega_{1})+\nabla_{G}\text{\rm -deg\,}(\nabla\varphi,\Omega_{2}).$

\item \textbf{(Homotopy)} If $\nabla_{x}\psi:[0,1]\times V\rightarrow V$ is a
$G$-gradient $\Omega$-admissible homotopy, then
\[
\nabla_{G}\text{\rm -deg\,}(\nabla_{x}\psi,\Omega
)=\text{\textit{constant}}.
\]

\item \textbf{(Normalization)} Let $\varphi\in C_{G}^{2}(V,\mathbb{R})$ be a
special $\Omega$-Morse function (cf. \cite{Geba}) such that $(\nabla\varphi)^{-1}(0)\cap
\Omega=G(u_{0})$ and $G_{u_{0}}=H$. Then,
\[
\nabla_{G}\text{\rm -deg\,}(\nabla\varphi,\Omega
)=(-1)^{\mathrm{m}^{-}(\nabla^{2}\varphi(u_{0}))}\cdot(H),
\]
where \textquotedblleft$\mathrm{m}^{-}(\cdot)$\textquotedblright\ stands for
the total dimension of all the eigenspaces corresponding to negative eigenvalues of a (symmetric) matrix.

\item \textbf{(Multiplicativity)} For all $(\nabla\varphi_{1},\Omega_{1})$,
$(\nabla\varphi_{2},\Omega_{2})\in\mathcal{M}_{\nabla}^{G}$,
\[
\nabla_{G}\text{\rm -deg\,}(\nabla\varphi_{1}\times\nabla
\varphi_{2},\Omega_{1}\times\Omega_{2})=\nabla_{G}\text{\rm -deg\,}(\nabla\varphi_{1},\Omega_{1})\ast\nabla_{G}\text{\textrm{-deg\thinspace}%
}(\nabla\varphi_{2},\Omega_{2})
\]
where the multiplication `$\ast$' is taken in the Euler ring $U(G)$.

\item \textbf{(Functoriality)}(cf. \cite{DaKr}) Suppose $G_o\le G$ is a subgroup of a compact Lie group $G$ such that $\text{\rm dim\,} G_0=\text{\rm dim\,} G$. Then any gradient admissible $G$-pair $(\nabla \vp,\Omega)$ is also an admissible $G_0$-pair and we have 
\[
\Psi\left[   \nabla_{G}\text{\rm -deg\,}(\nabla\varphi,\Omega).  \right]  = \nabla_{G_0}\text{\rm -deg\,}(\nabla\varphi,\Omega),
\]
where $\Psi:U(G)\to U(G_0)$ is the Euler ring homomorphism induced by the inclusion $\psi :G_0 \hookrightarrow G$ (see \cite{BKR}).

\end{description}
\end{theorem}

Using a standard finite-dimensional approximation scheme, the $G$-equivariant gradient degree can be extended to  admissible $G$-pairs in Hilbert $G$-representation. To be more precise, consider a Hilbert $G$-representation $\mathscr H$, a   $G$-equivariant completely continuous gradient field $\nabla f:\mathscr H\to \mathscr H$ and  an open bounded $G$-invariant set $\Omega\subset \mathscr H$, such that  $\nabla f$ is $\Omega$-admissible. Then the  pair $(\nabla f,\Omega)$ is called  a {\it $G$-admissible pair} in $\mathscr H$.  This degree admits  the same properties as those listed in Theorem \ref{thm:Ggrad-properties} (cf. \cite{survey,DaKr}).

One of the most important properties of $G$-equivariant gradient degree $\nabla_G\text{-deg}(\nabla f,\Omega)$ is that it provides a full equivariant topological classification of the solution set  for $\nabla f(x)=0$ and $x\in \Omega$. More precisely, in addition to the properties listed in Theorem \ref{thm:Ggrad-properties}, the equivariant gradient degree has also the so-called {\it Universality Property}, which says that two $B(V)$-admissible $G$-equivariant gradient maps $\nabla f_1$, $\nabla f_2:V\to V$ have the same gradient degrees if and only if they are $B(V)$-admissibly gradient homotopic.

Suppose that 
\[\nabla_G\text{-deg}(\nabla f,\Omega)=n_1(H_1)+n_2(H_2)+\dots +n_k(H_k)+\dots +n_m(H_m),
\]
and $n_k\not=0$. Then, by the existence property,  there exists
a solution $x_o\in \Omega$ of $\nabla f(x)=0$, such that $H:=G_{x_o}$. In addition, if $(H_k)$ is a maximal orbit type in $\Omega$, then for any $\Omega$-admissible continuous deformation $\{\nabla f_t\}_{t\in [0,1]}$ (in the class of gradient maps) we obtain a continuum of solutions in $\Omega$ to $\nabla f_t(x)=0$ that starts at $x_0$ for $t=0$ and ends at $x_1$ for $t=1$, and $G_{x_1}=H$. This property is called {\it Continuation Property}.

\subsection{Degree on the Slice}
Let $\mathscr H$ be a Hilbert $G$-representation.
Suppose that the orbit $G(u_{o})$ of $u_{o}\in\mathscr H$ is contained in a
finite-dimensional $G$-invariant subspace, so the $G$-action on that subspace
is smooth and $G(u_{o})$ is a smooth submanifold of $\mathscr H$. In such a case we  call the orbit $G(u_o)$ {\it finite-dimensional}. Denote by
$S_{o}\subset\mathscr H$ the slice to the orbit $G(u_{o})$ at $u_{o}$. Denote
by $V_{o}:=\tau_{u_{o}}G(u_{o})$ the tangent space to $G(u_{o})$ at $u_{o}$. Then
$S_{o}=V_{o}^{\perp}$ and $S_{o}$ is a smooth Hilbert $G_{u_{o}}$-representation.

Then we have (cf. \cite{BeKr})

\begin{theorem}
{\smc(Slice Principle)} \label{thm:SCP} Let $\mathscr{H}$ be a Hilbert
$G$-representation, $\Omega$ an open $G$-invariant subset in $\mathscr H$,  and $\varphi:\Omega\rightarrow\mathbb{R}$ a
continuously differentiable $G$-invariant functional such that $\nabla \vp$ is a completely continuous field. Suppose that  $u_{o}\in\Omega$  and  
$G(u_{o})$ is an finite-dimensional isolated critical orbit of $\varphi$ with  $S_{o}$ being the slice to the orbit $G(u_{o})$  at $u_o$, and $\mathcal{U}$ an isolated
tubular neighborhood of $G(u_{o})$. Put $\varphi_{o}:S_{o}\rightarrow
\mathbb{R}$ by $\varphi_{o}(v):=\varphi(u_{o}+v)$, $v\in S_{o}$. Then
\begin{equation}
\nabla_{G}\text{\rm -deg\,}(\nabla\varphi,\mathcal{U})=\Theta
(\nabla_{G_{u_{o}}}\text{\rm -deg\,}(\nabla\varphi_{o},\mathcal{U}\cap
S_{o})), \label{eq:SDP}%
\end{equation}
where $\Theta:U(G_{u_{o}})\rightarrow U(G)$ is homomorphism  defined on generators
$\Theta(H)=(H)$, $(H)\in\Phi(G_{u_{o}})$.
\end{theorem}

\subsection{$G$-Equivariant Gradient Degree of Linear Maps}

Let us establish a computational formula to evaluate the $G$-equivariant degree   $\nabla_{G}\text{-deg\thinspace}(\mathscr
A,B(V))$, where\textrm{ }$\mathscr A:V\rightarrow V$ is a symmetric $G
$-equivariant linear isomorphism and $V$ is an orthogonal $G$-representation,
i.e., $\mathscr A=\nabla\varphi$ for $\varphi(v)=\frac{1}{2}(\mathscr Av\bullet
v)$, $v\in V$. 
Consider the $G$-isotypical decomposition \eqref{eq:Giso}
of $V$ and put
\[
\mathscr A_{i}:=\mathscr A|_{V_{i}}:V_{i}\rightarrow V_{i},\quad
i=0,1,\dots,r.
\]
Then, by the multiplicativity property,
\begin{equation}
\nabla_{G}\mbox{-deg}(\mathscr A,B(V))=\prod_{i}^{r}\nabla_{G
}\mbox{-deg}(\mathscr A_{i},B(V_{i})) \label{eq:deg-Lin-decoGrad}%
\end{equation}
Take $\xi\in\sigma_{-}(\mathscr A)$, where $\sigma_{-}(\mathscr A)$ stands for
the negative spectrum of $\mathscr A$, and consider the corresponding
eigenspace $E(\xi):=\ker(\mathscr A-\xi\mbox{Id})$. Define the numbers
$m_{i}(\xi)$ by
\begin{equation}
m_{i}(\xi):=\dim\left(  E(\xi)\cap V_{i}\right)  /\dim\mathcal{V}_{i},
\label{eq:m_j(mu)-gra}%
\end{equation}
and the so-called {\it gradient $G$-equivariant basic degrees} by
\begin{equation}
\nabla\text{-deg}_{\mathcal{V}_{i}}:=\nabla_{G}%
\mbox{-deg}(-\mbox{Id\,},B(\mathcal{V}_{i})). \label{eq:basicGrad-deg0}, \quad i=0,1,2,\dots%
\end{equation}
Then
\begin{equation}
\nabla_{G}\text{\textrm{-deg}}(\mathscr A,B(V))=\prod_{\xi\in\sigma
_{-}(\mathscr A)}\prod_{i=0}^{r}\left(  {\nabla\mbox{\rm -deg}_{\mathcal{V}_{i}}%
}\right)  ^{m_{i}(\xi)}. \label{eq:lin-GdegGrad}%
\end{equation}

\subsection{$\Gamma\times O(2)$-Equivariant Basic Degrees}\label{sec:grad-techniques}

In order to be able to effectively use the formula \eqref{eq:deg-Lin-decoGrad}, it is important to establish the exact values of the gradient $G$-equivariant basic degrees. A direct usage of topological  definition (see \cite{Geba})  of the gradient $G$-equivariant  degree to compute the basic degrees  $\nabla\text{-deg}_{\mathcal V_i}$ (given by \eqref{eq:basicGrad-deg0}),    may be very complicated for infinite compact Lie groups $G$. However, in the case of the group $G:=\Gamma\times O(2)$ ($\Gamma$ being a finite group), we have  effective reduction techniques (see \cite{DaKr,RR}), using  the  homomorphism $\pi_0$ and the  Euler ring homomorphism $\Psi: U(\Gamma\times O(2))\to U(\Gamma\times S^1)$, which allow to establish the exact values of the gradient $\Gamma\times O(2)$-equivariant basic degrees.

To be more precise, let us recall the {\it $G$-equivariant Brouwer degree}\break
  $\gdeg( f,\Omega)\in A(G)$, which   is defined for admissible $G$-pairs $(f,\Omega)\in \mathcal M^G(V)$ and has similar existence, additivity, homotopy and multiplicativity properties  as the gradient degree. It can  be computed by applying the following recurrence formula to  the usual Brouwer degrees of maps $f^H:V^H\to V^H$, $(H)\in \Phi_0(G;V)$, i.e.,
\[  
\gdeg(f,\Omega)=\sum_{(H)\in \Phi_0(G;V)} n_H(H),
\]
and 
\begin{equation}\label{eq:rec-Brouwer}
n_H=\frac{\deg(f^H,\Omega^H)-\sum_{(L)>(H)} n_L \,n(H,L) \,|W(L)|}{|W(H)|}.
\end{equation}

In addition, for any gradient admissible $G$-pair $(\nabla\vp,\Omega)$, the  $G$-equivariant Brouwer  degree $\gdeg(\nabla \vp,\Omega)\in A(G)$ is well-defined  and we have  the following relation (see \cite{DaKr})
\begin{equation}\label{eq:degs}
\pi_0\left( \nabla_G\text{-deg}(\nabla \vp, \Omega)  \right)=\gdeg(\nabla \vp,\Omega).
\end{equation}
Moreover, the {\it  Brouwer $G$-equivariant basic degrees}
\[
\deg_{\mathcal V_i}:=\gdeg(-\id, B(\mathcal V_i)), \quad i=0,1,2,3,\dots.
\]
satisfy
\begin{equation}\label{eq:basic-pi0}
\deg_{\mathcal V_i}=\pi_0\left[ \nabla\text{-deg}_{\mathcal V_i}    \right], \quad i=0,1,2,3\dots
\end{equation}
Therefore, by \eqref{eq:basic-pi0},
\[ \nabla\text{-deg}_{\mathcal V_i}= \deg_{\mathcal V_i}+\sum_{(H)\in \Phi_1(G;\mathcal V_i)} x_H (H),\]
with the integers $x_H$ that need to be  determined by other means. 

Since the gradient  basic degrees  $\wt{\deg}_{\mathcal V_i}$ for the group $\Gamma\times S^1$ are well-known, one can apply the Euler ring homomorphism $\Psi:U(\Gamma\times O(2))\to U(\Gamma\times S^1)$ to determine these coefficients (see \cite{DaKr}) by using the relation 
\[
\wt{\deg}_{\mathcal V_i}=\Psi\left[ \deg_{\mathcal V_i}  \right]  +\sum_{(H)\in \Phi_1(G;\mathcal V_i)} x_H\Psi (H),
\] 
(here we assume that $S^1$ acts nontrivially on $\mathcal V_i$). The Euler ring homomorphism $\Psi: U(\Gamma\times O(2))\to U(\Gamma \times S^1)$ is defined on the generators by
\begin{align}
\Psi(H)&=
\begin{cases}
2(K) &\text{ if }\;\; K=H \text{ and } K\sim K' \text{ in } \Gamma\times SO(2),\\
(K)+(K')   &\text{ if }\;\; K=H \text{ and } K\not\sim K' \text{ in } \Gamma\times SO(2),\\
(K)  &\text{ if }\;\; K\not=H,
\end{cases}\label{eq:Euler-hom}
\end{align}
where $K:=H\cap \Gamma\times SO(2)$, $K':=\kappa H\kappa \cap \Gamma \times SO(2)$.

The formula \eqref{eq:rec-Brouwer} allows the usage of computational programs based on G.A.P. platform to obtain exact symbolic evaluation of the $G$-equivariant Brouwer degree of linear isomorphisms for a large class of classical groups and their products (see \cite{Pin}).
\vs

\subsection{Gradient $I\times O(2)$-Equivariant Basic Degrees}\label{sec:basic}

We used
GAP programming (all the GAP routines are available for download at the website listed in \cite{Pin}) to classify all the conjugacy classes
of closed subgroups in $G:=I\times O(2)$. We also computed the following basic gradient 
degrees (corresponding to the irreducible $G$-representations associated to the characters listed in Table \ref{tab:I}), where we use red color to indicate the maximal orbit types (in the class of $2\pi$-periodic functions $x:\br\to V$ with Fourier mode $1$).

{\small
\begin{align*}
\eqdeg{\nabla}_{\cV_{1,1}}=\; &  -{\color{red}(\amal{{A_5^p}}{D_{1}}{}{}{})}%
+(\amal{{A_5^p}}{O(2)}{}{}{}),\\
\eqdeg{\nabla}_{\cV_{2,1}}=\; &
-{\color{red}(\amal{{D_3^p}}{D_{2}}{\bz_{2}}{{\bz_3^p}}{})}%
-{\color{red}(\amal{{V_4^p}}{D_{2}}{\bz_{2}}{{\bz_2^p}}{})}%
+2(\amal{{\bz_2^p}}{D_{2}}{\bz_{2}}{{\bz_1^p}}{})-{\color{red}%
(\amal{{A_4^p}}{D_{1}}{}{}{})}\\
&  -{\color{red}
(\amal{{D_3^p}}{D_{1}}{}{}{})}+2(\amal{{\bz_3^p}}{D_{1}}{}{}{})+2(\amal{{\bz_2^p}}{D_{1}}{}{}{})-2(\amal{{\bz_1^p}}{D_{1}}{}{}{})\\
&  -{\color{red}(\amal{{D_5^p}}{D_{5}}{D_{5}}{{\bz_1^p}}{{1}})}%
-{\color{red}(\amal{{D_5^p}}{D_{5}}{D_{5}}{{\bz_1^p}}{{2}})}%
-{\color{red}(\amal{{D_3^p}}{D_{3}}{D_{3}}{{\bz_1^p}}{})}%
+(\amal{{V_4^p}}{D_{2}}{D_{2}}{{\bz_1^p}}{})\\
&
+(\amal{{D_3^p}}{D_{1}}{D_{1}}{{\bz_3^p}}{})+(\amal{{V_4^p}}{D_{1}}{D_{1}}{{\bz_2^p}}{})+(\amal{{A_5^p}}{O(2)}{}{}{})\\
&
-(\amal{\bz_3^p}{\bz_1}{}{}{})-(\amal{\bz_2^p}{\bz_1}{}{}{})+(\amal{\bz_1^p}{\bz_1}{}{}{}),\\
\eqdeg{\nabla}_{\cV_{3,1}}=\; &
-{\color{red}(\amal{{V_4^p}}{D_{2}}{\bz_{2}}{{\bz_2^p}}{})}%
+(\amal{{\bz_2^p}}{D_{2}}{\bz_{2}}{{\bz_1^p}}{})-{\color{red}
(\amal{{D_5^p}}{D_{1}}{}{}{})}-{\color{red}(\amal{{D_3^p}}{D_{1}}{}{}{})}\\
&
+2(\amal{{\bz_2^p}}{D_{1}}{}{}{})-(\amal{{\bz_1^p}}{D_{1}}{}{}{})-{\color{red}(\amal{{D_5^p}}{D_{5}}{D_{5}}{{\bz_1^p}}{{1}})}%
-{\color{red}(\amal{{D_5^p}}{D_{5}}{D_{5}}{{\bz_1^p}}{{2}})}\\
&
+(\amal{{V_4^p}}{D_{2}}{D_{2}}{{\bz_1^p}}{})+(\amal{{\bz_2^p}}{D_{1}}{D_{1}}{{\bz_1^p}}{})+(\amal{{A_5^p}}{O(2)}{}{}{})\\
&  -{\color{red}%
(\amal{A_4^p}{\bz_{3}}{\bz_3}{V_4^p}{})}-(\amal{\bz_2^p}{\bz_{2}}{\bz_2}{\bz_1^p}{}),\\
\eqdeg{\nabla}_{\cV_{4,1}}=\; &
-{\color{red}(\amal{{D_5^p}}{D_{2}}{\bz_{2}}{{\bz_5^p}}{})}%
-{\color{red}(\amal{{D_3^p}}{D_{2}}{\bz_{2}}{{\bz_3^p}}{})}%
-{\color{red}(\amal{{V_4^p}}{D_{2}}{\bz_{2}}{{\bz_2^p}}{})}%
+3(\amal{{\bz_2^p}}{D_{2}}{\bz_{2}}{{\bz_1^p}}{})\\
&
-(\amal{{\bz_1^p}}{D_{1}}{}{}{})-{\color{red}(\amal{{D_5^p}}{D_{5}}{D_{5}}{{\bz_1^p}}{{1}})}%
-{\color{red}(\amal{{D_3^p}}{D_{3}}{D_{3}}{{\bz_1^p}}{})}%
+(\amal{{V_4^p}}{D_{2}}{D_{2}}{{\bz_1^p}}{})\\
&
+(\amal{{\bz_2^p}}{D_{1}}{D_{1}}{{\bz_1^p}}{})+(\amal{{A_5^p}}{O(2)}{}{}{})  -(\amal{\bz_2^p}{\bz_{2}}{\bz_2}{\bz_1^p}{}),\\
\eqdeg{\nabla}_{\cV_{5,1}}=\; &
-{\color{red}(\amal{{D_5^p}}{D_{2}}{\bz_{2}}{{\bz_5^p}}{})}%
-{\color{red}(\amal{{D_3^p}}{D_{2}}{\bz_{2}}{{\bz_3^p}}{})}%
-{\color{red}(\amal{{V_4^p}}{D_{2}}{\bz_{2}}{{\bz_2^p}}{})}%
+3(\amal{{\bz_2^p}}{D_{2}}{\bz_{2}}{{\bz_1^p}}{})\\
&
-(\amal{{\bz_1^p}}{D_{1}}{}{}{})-{\color{red}(\amal{{D_5^p}}{D_{5}}{D_{5}}{{\bz_1^p}}{{2}})}%
-{\color{red}(\amal{{D_3^p}}{D_{3}}{D_{3}}{{\bz_1^p}}{})}%
+(\amal{{V_4^p}}{D_{2}}{D_{2}}{{\bz_1^p}}{})\\
&
+(\amal{{\bz_2^p}}{D_{1}}{D_{1}}{{\bz_1^p}}{})+(\amal{{A_5^p}}{O(2)}{}{}{})  -(\amal{\bz_2^p}{\bz_{2}}{\bz_2}{\bz_1^p}{})\\
\eqdeg{\nabla}_{\cV_{-1,1}}=\; &
-{\color{red}(\amal{{A_5^p}}{D_{2}}{\bz_{2}}{{A_5}}{})}%
+(\amal{{A_5^p}}{O(2)}{}{}{}),\\
\eqdeg{\nabla}_{\cV_{-2,1}}=\; &  -{\color{red}%
(\amal{{A_4^p}}{D_{2}}{\bz_{2}}{{A_4}}{})}-{\color{red}(\amal{{D_3^p}}{D_{2}}{\bz_{2}}{{D_3^z}}{})}%
-{\color{red}%
(\amal{{D_3^p}}{D_{2}}{\bz_{2}}{{D_3}}{})}-{\color{red}(\amal{{V_4^p}}{D_{2}}{\bz_{2}}{{V_4^z}}{})}%
\\
&
+2(\amal{{\bz_3^p}}{D_{2}}{\bz_{2}}{{\bz_3}}{})+2(\amal{{\bz_2^p}}{D_{2}}{\bz_{2}}{{\bz_2^z}}{})+2(\amal{{\bz_2^p}}{D_{2}}{\bz_{2}}{{\bz_2}}{})-2(\amal{{\bz_1^p}}{D_{2}}{\bz_{2}}{{\bz_1}}{})\\
&  -{\color{red}(\amal{{D_5^p}}{D_{10}}{D_{10}}{{\bz_1}}{{1}})}%
-{\color{red}(\amal{{D_5^p}}{D_{10}}{D_{10}}{{\bz_1}}{{2}})}%
-{\color{red}(\amal{{D_3^p}}{D_{6}}{D_{6}}{{\bz_1}}{})}%
+(\amal{{D_3^p}}{D_{2}}{D_{2}}{{\bz_3}}{{\bz_3^p}})\\
&
+(\amal{{V_4^p}}{D_{2}}{D_{2}}{{\bz_2^z}}{{\bz_2^p}})+(\amal{{V_4^p}}{D_{2}}{D_{2}}{{\bz_2}}{{\bz_2^p}})+(\amal{{A_5^p}}{O(2)}{}{}{})  +(\amal{\bz_1^p}{\bz_{2}}{\bz_2}{\bz_1}{}),\\
\eqdeg{\nabla}_{\cV_{-3,1}}=\; &  -{\color{red}%
(\amal{{D_5^p}}{D_{2}}{\bz_{2}}{{D_5}}{})}-{\color{red}%
(\amal{{D_3^p}}{D_{2}}{\bz_{2}}{{D_3}}{})}-{\color{red}(\amal{{V_4^p}}{D_{2}}{\bz_{2}}{{V_4^z}}{})}%
+(\amal{{\bz_2^p}}{D_{2}}{\bz_{2}}{{\bz_2^z}}{})\\
&
+2(\amal{{\bz_2^p}}{D_{2}}{\bz_{2}}{{\bz_2}}{})-(\amal{{\bz_1^p}}{D_{2}}{\bz_{2}}{{\bz_1}}{})-{\color{red}(\amal{{D_5^p}}{D_{10}}{D_{10}}{{\bz_1}}{{1}})}%
-{\color{red}(\amal{{D_5^p}}{D_{10}}{D_{10}}{{\bz_1}}{{2}})}\\
&
+(\amal{{V_4^p}}{D_{2}}{D_{2}}{{\bz_2^z}}{{\bz_2^p}})+(\amal{{\bz_2^p}}{D_{2}}{D_{2}}{{\bz_1}}{{\bz_1^p}})+(\amal{{A_5^p}}{O(2)}{}{}{})\\
&  -{\color{red}(\amal{A_4^p}{\bz_{6}}{\bz_6}{V_4}{})}%
-(\amal{\bz_2^p}{\bz_{2}}{\bz_2}{\bz_2^z}{}),
\end{align*}
\bigskip

\begin{align*}
\eqdeg{\nabla}_{\cV_{-4,1}}=\; &
-{\color{red}(\amal{{D_5^p}}{D_{2}}{\bz_{2}}{{D_5^z}}{})}%
-{\color{red}(\amal{{D_3^p}}{D_{2}}{\bz_{2}}{{D_3^z}}{})}%
-{\color{red}(\amal{{V_4^p}}{D_{2}}{\bz_{2}}{{V_4^z}}{})}%
+3(\amal{{\bz_2^p}}{D_{2}}{\bz_{2}}{{\bz_2^z}}{})\\
&
-(\amal{{\bz_1^p}}{D_{2}}{\bz_{2}}{{\bz_1}}{})-{\color{red}(\amal{{D_5^p}}{D_{10}}{D_{10}}{{\bz_1}}{{1}})}%
-{\color{red}(\amal{{D_3^p}}{D_{6}}{D_{6}}{{\bz_1}}{})}%
+(\amal{{V_4^p}}{D_{2}}{D_{2}}{{\bz_2^z}}{{\bz_2^p}})\\
&
+(\amal{{\bz_2^p}}{D_{2}}{D_{2}}{{\bz_1}}{{\bz_1^p}})+(\amal{{A_5^p}}{O(2)}{}{}{})  -(\amal{\bz_2^p}{\bz_{2}}{\bz_2}{\bz_2^z}{}),\\
\eqdeg{\nabla}_{\cV_{-5,1}}=\; &
-{\color{red}(\amal{{D_5^p}}{D_{2}}{\bz_{2}}{{D_5^z}}{})}%
-{\color{red}(\amal{{D_3^p}}{D_{2}}{\bz_{2}}{{D_3^z}}{})}%
-{\color{red}(\amal{{V_4^p}}{D_{2}}{\bz_{2}}{{V_4^z}}{})}%
+3(\amal{{\bz_2^p}}{D_{2}}{\bz_{2}}{{\bz_2^z}}{})\\
&
-(\amal{{\bz_1^p}}{D_{2}}{\bz_{2}}{{\bz_1}}{})-{\color{red}(\amal{{D_5^p}}{D_{10}}{D_{10}}{{\bz_1}}{{2}})}%
-{\color{red}(\amal{{D_3^p}}{D_{6}}{D_{6}}{{\bz_1}}{})}%
+(\amal{{V_4^p}}{D_{2}}{D_{2}}{{\bz_2^z}}{{\bz_2^p}})\\
&
+(\amal{{\bz_2^p}}{D_{2}}{D_{2}}{{\bz_1}}{{\bz_1^p}})+(\amal{{A_5^p}}{O(2)}{}{}{}) -(\amal{\bz_2^p}{\bz_{2}}{\bz_2}{\bz_2^z}{}),
\end{align*}
}

Let us point out that the computational algorithms for the basic degrees truncated to $A(I\times O(2))$ where established in \cite{DaKr}, and are now being effectively implemented into a computer software for equivariant gradient degree. Most of the maximal orbit types of periodic vibrations have finite Weyl groups, i.e., they are in the basic gradient degrees truncated to $A(I\times O(2))$. Nevertheless,  we have detected in the basic gradient degree of the representations $\mathcal{V}_{\pm3,1}$ a maximal group with Weyl group of dimension one. 

\bigskip
\noindent\textbf{Acknowledgement.}  C. Garc\'{\i}a was partially
supported by PAPIIT-UNAM through grant IA105217. W. Krawcewicz acknowledge partial support from National Science Foundation through grant DMS-1413223 and from National Science Foundation of China through grant no. 11871171.
\vs

\end{document}